\documentclass[pdflatex,sn-mathphys-num]{sn-jnl}

\usepackage{amsmath,amssymb,amsthm,mathtools}
\usepackage{graphicx}%
\usepackage{multirow}%
\usepackage{amsmath,amssymb,amsfonts}%
\usepackage{amsthm}%
\usepackage{mathrsfs}%
\usepackage[title]{appendix}%
\usepackage{xcolor}%
\usepackage{textcomp}%
\usepackage{manyfoot}%
\usepackage{booktabs}%
\usepackage{algorithm}%
\usepackage{algorithmicx}%
\usepackage{algpseudocode}%
\usepackage{listings}%
\usepackage{xcolor}
\usepackage{graphicx}
\usepackage{booktabs}
\usepackage{hyperref}
\hypersetup{
  colorlinks=true,
  linkcolor=blue!55!black,
  citecolor=blue!55!black,
  urlcolor=blue!55!black
}

\newcommand{\R}{\mathbb{R}}
\newcommand{\Om}{\Omega}
\newcommand{\dd}{\,\mathrm{d}}
\newcommand{\Uad}{\mathcal{U}_{ad}}
\newcommand{\UadBV}{\mathcal{U}_{ad}^{BV}}
\DeclareMathOperator*{\argmax}{arg\,max}
\DeclareMathOperator{\TV}{TV}
\DeclareMathOperator{\sgn}{sgn}

\theoremstyle{plain}
\newtheorem{theorem}{Theorem}[section]
\newtheorem{proposition}[theorem]{Proposition}
\newtheorem{lemma}[theorem]{Lemma}
\newtheorem{corollary}[theorem]{Corollary}
\theoremstyle{definition}
\newtheorem{definition}[theorem]{Definition}
\newtheorem{assumption}[theorem]{Assumption}
\theoremstyle{remark}
\newtheorem{remark}[theorem]{Remark}

\begin{document}

\title{Vital-Rate Feedback and Threshold Harvesting in Size-Structured Populations}

\author{\fnm{Louis Shuo} \sur{Wang}}\email{wang.s41@northeastern.edu}

\affil[2]{\orgdiv{Department of Mathematics},
\orgname{Northeastern University},
\orgaddress{
\city{Boston},
\state{MA},
\postcode{02115},
\country{USA}
}}

\abstract{
Size-selective harvesting is often justified by the intuition that larger
individuals have higher value and should therefore be harvested above a critical
size. We study a controlled size-structured transport
model in which a scalar environmental variable, generated by the population
itself, modifies the vital rates \(g(E,l)\) and \(\mu(E,l)\). The first result is
a corrected stationary closure theory: crowding-suppressed growth increases
residence density at the inflow, so the closure derivative is not determined by
pointwise profile monotonicity. Instead,
$\Phi'(E)=\mathsf A(E)-\mathsf C(E)$,
an integrated balance between residence-time amplification and cumulative
survival loss. The second result is an exact stationary adjoint reduction. The
nonlocal switching correction has rank one,
$S=S_{\rm red}-\frac{A}{1-B}\psi$,
and its zero-discount feedback gain satisfies the identity
$B(0)=\Phi'(E^*)$.
Thus the same scalar governs stationary closure sensitivity and threshold
fragility. We also establish finite-horizon well-posedness and compactified
optimal-control existence in a spatial-\(BV\) policy class. Numerical
certification in a density-dependent von Bertalanffy model shows when
minimum-size harvesting persists and when vital-rate feedback creates
multiple-switch harvest windows.}



\maketitle

% ======================================================================
\section{Introduction}\label{sec:intro}

\subsection{Existing theory: structured populations and selective harvesting}
\label{subsec:existing_theory}

Individual size governs growth, survival, reproductive output, susceptibility to
harvesting gear, and economic value. Population models that retain size as a
continuous physiological coordinate are therefore a natural framework for
selective harvesting. In such models, the population density is transported
through size by individual growth, depleted by natural and harvesting mortality,
and replenished through an inflow or renewal boundary. The resulting dynamics
belong to the general theory of physiologically structured populations
\cite{metz2014dynamics,diekmann2003steady,cai2026optimal,gyllenberg2007mathematical,
diekmann2025age,yu2026rigorous,bouguima2023analysis,chen2022maximum}.

A central feature of this theory is the environmental interaction variable:
individuals alter a finite-dimensional environmental quantity, while that
quantity modifies their vital rates. At steady state, the problem separates into
an individual-level balance under a frozen environment and a population-level
consistency condition that closes the feedback loop
\cite{diekmann2003steady}. For a scalar environment $E$, this frequently
leads to a closure equation of the form $E=\Phi(E)$. Scalar reductions, their
monotonicity properties, and their relation to persistence are classical
\cite{metz2014dynamics,webb1985theory,iannelli2017basic,cushing1994structured,thieme2018mathematics,wang2025analysis,gyllenberg2007mathematical}. The corresponding linearization around a steady state is also
known to have finite-dimensional feedback structure in the environmental
coordinate, with characteristic equations obtained from scalar or
finite-dimensional closure loops
\cite{diekmann1998formulation,diekmann2001formulation,wang2025analysis1,
diekmann2003steady}. Thus the use of an aggregate crowding
variable, the reduction of a stationary structured model to a scalar consistency
equation, and the finite-rank character of the environmental feedback
linearization form part of the established structured-population framework.

Optimal harvesting of continuously structured populations is likewise a mature
subject, with deep roots in fisheries bioeconomics
\cite{beverton2012dynamics,clark2010mathematical,opmeer2025optimal}. Early work on age-structured
harvesting established optimality conditions and characteristic-based control
formulations for McKendrick--von Foerster equations
\cite{gurtin1981optimal,murphy1990optimal,liu2025bidirectional,brokate1985pontryagin,
anita2013analysis,barbu1999optimal,feichtinger2003optimality,pontryagin2018mathematical}.
For size-structured populations, Kato proved existence of an optimal harvesting
rate for a nonlinear model with separable mortality and nonlinear fertility, and
separately derived a maximum principle and sufficient bang--bang conditions for a
linear size-structured model \cite{kato2008maximum,kato2008optimal}. Hritonenko, Yatsenko,
Goetz and Xabadia established maximum-principle and bang--bang results for
age- and size-structured harvesting models, including size-structured forest
management \cite{hritonenko2009bang}. Davydov and Platov studied stationary
profit maximization when growth, mortality, and exploitation depend on size
alone, obtaining existence, uniqueness, and necessary optimality conditions under
explicit parameter assumptions \cite{davydov2011optimization}.

Subsequent work has broadened the theory to include size-dependent areas of
action, spatial structure, discontinuous controls, hierarchical competition,
nonlocal boundary feedback, impulse controls, and measure-valued controls
\cite{anita2019optimal,anicta2019optimal,liang2025global,
hritonenko2012bang,coclite2017time,calsina1995model,wang2026damage,
cushing1994structured}. Measure-valued and balance-law formulations provide compactness
frameworks when classical controls or densities develop concentrations or rapid
oscillations \cite{bressan2013multidimensional,yu2026pattern,coclite2017time,
carrillo2012structured,debiec2018relative,wang2026algebraic}. Recent studies continue this line by
considering density dependence, arbitrary length-dependent harvest schedules,
generic cost functionals, and alternative optimal-control formulations
\cite{opmeer2025optimal,kakumani2026optimal}; broader optimal-control theory for
age-structured dynamic models provides further context
\cite{freiberger2025optimization,kang2024dynamical,khan2022optimal,wu2025age,chen2025numerical,wang2023global}. In particular, Kakumani and Tumuluri
\cite{kakumani2026optimal} treat optimal harvesting for a nonlinear
McKendrick--von Foerster equation with a generic cost functional.

The closest work to the present application is that of Opmeer
\cite{opmeer2025optimal}, which permits an arbitrary length-dependent harvesting schedule
in a density-dependent age/length formulation and numerically finds a conventional
minimum-size optimum. The present paper differs in its analytical focus. Here the
environmental feedback enters directly through the size-transport vital rates
$g(E,l)$ and $\mu(E,l)$, rather than only through recruitment, boundary feedback,
or a size-only exploitation term. This changes both the stationary closure and
the harvesting switching function. First, density-suppressed growth increases the
residence density at the inflow, so the sign of the stationary closure derivative
cannot be inferred from pointwise profile monotonicity; it is determined by an
integrated balance of residence-time amplification and cumulative survival loss.
Second, the stationary adjoint receives a full nonlocal correction that reduces
exactly to a rank-one term. Its zero-discount gain is the stationary closure
derivative,
$\displaystyle B|_{r=0}=\Phi'(E^*)$.

The paper therefore has two main analytical contributions and three supporting
components. The main contributions are: (i) a corrected stationary closure
criterion showing why vital-rate feedback destroys pointwise monotonicity but can
still yield an integrated uniqueness certificate; and (ii) an exact rank-one
stationary switching correction, together with the identity
$B|_{r=0}=\Phi'(E^*)$, which links closure sensitivity and adjoint feedback gain.
The supporting components are: (iii) finite-horizon well-posedness in a
spatial-$BV$ policy class; (iv) compactified optimal-control existence for
bounded-complexity size-selective policies; and (v) numerical certification in a
density-dependent von Bertalanffy model, including regimes where a conventional
minimum-size rule persists and regimes where it fails.

\subsection{Model and point of departure}
\label{subsec:point_departure}

Let
\[
\Om=[l_0,l_m],\qquad 0<l_0<l_m<\infty,
\]
be a bounded physiological-size domain. The population density
$x=x(t,l)\ge0$ satisfies the controlled transport equation
\cite{perthame2007transport,yu2026rigorous}
\begin{equation}\label{eq:intro_pde}
\partial_t x+\partial_l\!\big(g(E(t),l)x\big)
=
-\big(\mu(E(t),l)+u(t,l)\big)x,
\qquad
E(t)=\int_{l_0}^{l_m}\chi(l)x(t,l)\dd l,
\end{equation}
with initial datum $x(0,\cdot)=\phi$ and prescribed inflow flux
$\displaystyle g(E(t),l_0)x(t,l_0)=p(t)$.
Here $g(E,l)>0$ is the somatic growth rate, $\mu(E,l)\ge0$ is natural mortality,
$u(t,l)$ is size-specific harvesting mortality, and $E(t)$ is a weighted crowding
or resource-pressure variable. Depending on the weight $\chi$, the environment
may represent abundance, biomass, metabolic demand, or a more general load. We
assume the compensatory signs
\begin{equation}\label{eq:intro_signs}
\partial_E g(E,l)\le0,
\qquad
\partial_E\mu(E,l)\ge0,
\end{equation}
so that crowding slows growth and increases mortality.

The harvesting control satisfies $0\le u(t,l)\le u_{\max}$ a.e.\ on
$(0,T)\times\Om$, and the finite-horizon discounted yield is
\begin{equation}\label{eq:intro_J}
J_T(u)
=
\int_0^T e^{-rt}
\int_{l_0}^{l_m}c(l)u(t,l)x(t,l)\dd l\dd t,
\qquad r>0.
\end{equation}
Because the payoff is affine in $u$, a maximum-principle argument naturally leads
to bang--bang harvesting. The central structural question is stronger: when is
the bang--bang law also a minimum-size policy,
$\displaystyle u(t,l)=u_{\max}\mathbf 1_{\{l>L(t)\}}$?
A bang--bang law only asserts that $u^*(t,l)\in\{0,u_{\max}\}$ a.e.; a
minimum-size law additionally requires the harvested set to be an upper interval
in size. If the switching function is
$\displaystyle S(t,l)=c(l)-\lambda(t,l)$,
then bang--bang control follows from the sign of $S$, whereas a threshold policy
requires an upward single crossing of $S(t,\cdot)$. Global monotonicity of
$S(t,\cdot)$ is sufficient, but not necessary, for this property.

The environmental feedback in \eqref{eq:intro_pde} changes the stationary closure
in a simple but important way. For a stationary environment $E$ and stationary
control $u=u(l)$, the frozen-$E$ profile is
\[
x_E(l)
=
\frac{p}{g(E,l)}
\exp\!\left(
-\int_{l_0}^{l}
\frac{\mu(E,\xi)+u(\xi)}{g(E,\xi)}
\dd\xi
\right),
\]
and stationarity requires the scalar consistency condition
\[
E=\Phi(E),
\qquad
\Phi(E):=
\int_{l_0}^{l_m}\chi(l)x_E(l)\dd l.
\]
Although \eqref{eq:intro_signs} suggests that crowding should lower the profile,
pointwise monotonicity fails at the inflow. If $\chi(l_0)>0$, then
\[
\partial_E\log\!\big(\chi(l_0)x_E(l_0)\big)
=
-\frac{\partial_Eg(E,l_0)}{g(E,l_0)}
\ge0,
\]
with strict inequality whenever growth is genuinely suppressed at $l_0$. The
reason is that the boundary prescribes entering flux, not entering density:
slower growth raises the residence density $p/g(E,l_0)$ of newly entering
individuals. Farther along the size domain, increased mortality and delayed
progression reduce survival. Hence the derivative of the closure map is governed
by an integrated balance
$\displaystyle \Phi'(E)=\mathsf A(E)-\mathsf C(E)$,
where $\mathsf A(E)$ is residence-time amplification and $\mathsf C(E)$ is
cumulative survival/progression loss. This balance, rather than a pointwise sign
condition on $\partial_E x_E$, is the relevant stationary uniqueness criterion.

The same vital-rate feedback also produces a nonlocal adjoint correction. Since
\[
\delta E(t)=\int_{l_0}^{l_m}\chi(\xi)\,\delta x(t,\xi)\dd\xi,
\]
variations of $g(E,l)$ and $\mu(E,l)$ generate the formal source
\begin{equation}\label{eq:intro_adjsource}
\chi(l)
\int_{l_0}^{l_m}
\Big(
\partial_Eg(E,\xi)\,\partial_\xi\lambda(\xi)
-
\partial_E\mu(E,\xi)\lambda(\xi)
\Big)
x(\xi)\dd\xi .
\end{equation}
This is the mechanism absent from size-only adjoints or from models in which
feedback enters only through a recruitment or boundary functional. In the
spatial-$BV$ setting appropriate for finite-switch bang--bang policies, the term
involving $\partial_\xi\lambda$ must be interpreted by a $BV$
integration-by-parts identity, which converts \eqref{eq:intro_adjsource} into a
bounded zeroth-order nonlocal functional of the costate.

In the stationary canonical problem this nonlocality admits an exact scalar
reduction. Let $\lambda_{\mathrm{red}}$ be the reduced adjoint obtained by
freezing the environment and omitting the nonlocal sensitivity term, and let
$\psi$ be the response to one unit of scalar environmental source. Then the full
stationary adjoint has the form
\[
\lambda
=
\lambda_{\mathrm{red}}+\Gamma\psi,
\qquad
\Gamma=\frac{A}{1-B},
\]
where $A$ and $B$ are explicit scalar integrals at the operating point. Hence
\[
S=S_{\mathrm{red}}-\Gamma\psi,
\qquad
S_{\mathrm{red}}:=c-\lambda_{\mathrm{red}}.
\]
The threshold question is thereby reduced to the sign geometry of a computable
rank-one correction. The main stationary identity proved below is
$\displaystyle B|_{r=0}=\Phi'(E^*)$,
so the zero-discount adjoint feedback gain is exactly the stationary closure
derivative. This identity is the link between the corrected closure criterion and
the preservation or failure of minimum-size harvesting.
\subsection{Main contributions}
\label{subsec:main_contributions}

The paper has two main analytical contributions and three supporting
components. The first main contribution is a corrected stationary closure
criterion for size-structured populations whose environmental feedback enters
the vital rates \(g(E,l)\) and \(\mu(E,l)\). We show that density-suppressed growth
raises the residence density at the inflow, so pointwise profile monotonicity
fails; uniqueness is instead governed by the integrated balance
$\displaystyle \Phi'(E)=\mathsf A(E)-\mathsf C(E)$.
The second main contribution is an exact rank-one reduction of the stationary
nonlocal adjoint:
\[
\lambda=\lambda_{\mathrm{red}}+\frac{A}{1-B}\psi,
\qquad
S=S_{\mathrm{red}}-\frac{A}{1-B}\psi,
\]
together with the identity
$\displaystyle B(0)=\Phi'(E^*)$.
Thus the zero-discount adjoint feedback gain is precisely the stationary closure
derivative.

The supporting contributions are: finite-horizon well-posedness in a spatial-\(BV\)
policy class; compactified optimal-control existence for bounded-complexity
size-selective harvesting policies; and numerical certification in a
density-dependent von Bertalanffy model, showing when minimum-size harvesting is
preserved and when vital-rate feedback produces multiple-switch harvest windows.

\subsection{Organization of the paper}
\label{subsec:organization}

Section~\ref{sec:model} introduces the controlled size-structured model, the
environmental feedback variable, and the harvesting objective. 
Section~\ref{sec:stationary} derives the stationary closure equation
$E=\Phi(E)$ and the corrected integrated monotonicity criterion for uniqueness.
Section~\ref{sec:wellposed} proves finite-horizon well-posedness in the
spatial-$BV$ policy class. Section~\ref{sec:persistence} separates forced
persistence under prescribed recruitment from intrinsic demographic replacement.
Section~\ref{sec:control_existence} establishes compactified optimal-control
existence. Section~\ref{sec:adjoint} derives the nonlocal adjoint and the
corresponding variational inequality, and Section~\ref{sec:threshold} explains
when bang--bang harvesting becomes a minimum-size policy. 
Section~\ref{sec:stationary_switching} contains the main stationary switching
result: the exact rank-one correction and the identity
$B|_{r=0}=\Phi'(E^*)$. Section~\ref{sec:concrete} illustrates the theory
in a density-dependent von Bertalanffy model and maps the regimes in which
minimum-size harvesting persists or fails.
Finally, Appendix provides detailed proofs for the mathematical conclusions in the main text.
% ============================================================================
\section{Model and environmental feedback}\label{sec:model}

\subsection{Population state and physiological-size domain}
\label{subsec:state_space}

For each $t\ge0$, the state $x(t,\cdot)\in L^1_+(\Om)$ is the population-size
density, so that $\int_a^b x(t,l)\dd l$ is the number of individuals whose sizes
lie in $(a,b)\subseteq\Om$. In particular,
$\displaystyle N(t):=\int_{l_0}^{l_m}x(t,l)\dd l$
is total abundance. The lower endpoint $l_0$ is the recruitment size; the upper
endpoint $l_m$ is an exit size. Because growth is strictly positive, $l_0$ is an
inflow boundary and $l_m$ an outflow boundary. The model retains size but not
age: individuals of the same current size share the same instantaneous rates at a
given environmental state.

\subsection{Environmental interaction variable}
\label{subsec:environment}

Density dependence is mediated by the scalar environmental interaction variable
\begin{equation}\label{eq:E}
E(t):=\int_{l_0}^{l_m}\chi(l)x(t,l)\dd l,
\qquad
\chi\in L^\infty_+(\Om).
\end{equation}
Since $0\le E(t)\le \|\chi\|_{L^\infty(\Om)}N(t)$, the environment is controlled by total abundance but need not coincide with it. The functional \eqref{eq:E} is the finite-dimensional environmental closure of the structured model \cite{metz2014dynamics,diekmann2003steady}, separating the population--environment interaction into a readout $x\mapsto E=\int_\Om\chi x\,\dd l$ and a feedback $E\mapsto(g(E,\cdot),\mu(E,\cdot))$. Different choices of $\chi$ encode different mechanisms: $\chi\equiv1$ gives $E=N$; $\chi(l)=l$ gives a length-weighted load; $\chi(l)=\alpha l^\beta$ approximates a biomass or metabolic-demand load.

The scalar closure is deliberately parsimonious: it represents competition for a shared limiting resource without an explicit resource equation, on the assumption that the relevant environmental response is fast or well summarized by $E[x]$.

\subsection{Growth and natural mortality}
\label{subsec:vital_rates}

Let $g:\R_+\times\Om\to\R_+$ and $\mu:\R_+\times\Om\to\R_+$ be the growth and
mortality rates. We assume
\begin{align}
&g,\mu\in C^1(\R_+\times\Om),
\label{eq:H1}\\
&\text{for every }M>0:\quad
0<g_{\min}(M)\le g(E,l)\le g_{\max}(M)<\infty
\ \text{ on }[0,M]\times\Om,
\label{eq:H2}\\
&\mu(E,l)\ge\mu_{\min}\ge0,
\label{eq:H3}\\
&\partial_Eg(E,l)\le0,
\qquad
\partial_E\mu(E,l)\ge0,
\label{eq:H4}\\
&\partial_Eg,\ \partial_lg,\ \partial_E\mu,\ \partial_l\mu
\text{ are locally bounded on }\R_+\times\Om.
\label{eq:H5}
\end{align}
Assumption \eqref{eq:H2} is the bounded-environment positivity bound: on
each bounded environmental interval $[0,M]$ growth is strictly positive, so
characteristics move from $l_0$ toward $l_m$ and no boundary condition is needed
at $l_m$. We use it on the invariant interval $[0,M_T]$ in the finite-horizon
analysis (Section~\ref{sec:wellposed}) and on a stationary closure interval
$[0,M_{\mathrm{st}}]$ in Section~\ref{sec:stationary}; the calibrated growth law
of Section~\ref{sec:concrete} satisfies \eqref{eq:H2} in this local form (its
infimum over all $E\ge0$ is zero, but it is bounded below on every $[0,M]$). The
signs \eqref{eq:H4} describe compensatory density dependence. Slower growth
increases residence time within a size interval, which can locally increase
$x(t,l)$; increased mortality decreases survival to larger sizes. The competition
between residence-time accumulation and mortality loss is the source of the
corrected integrated closure analysis of Section~\ref{sec:stationary}.

As shown in Figure~\ref{fig:characteristics}, the critical characteristic (dashed line) partitions the phase plane into the initial-datum-fed region and the boundary-flux-fed region. The typical flattening of the curves near $l_0$ directly reflects the mathematical profile of the growth rate $g(E, l)$.

\begin{figure}[htbp]
    \centering
    \includegraphics[width=0.75\linewidth]{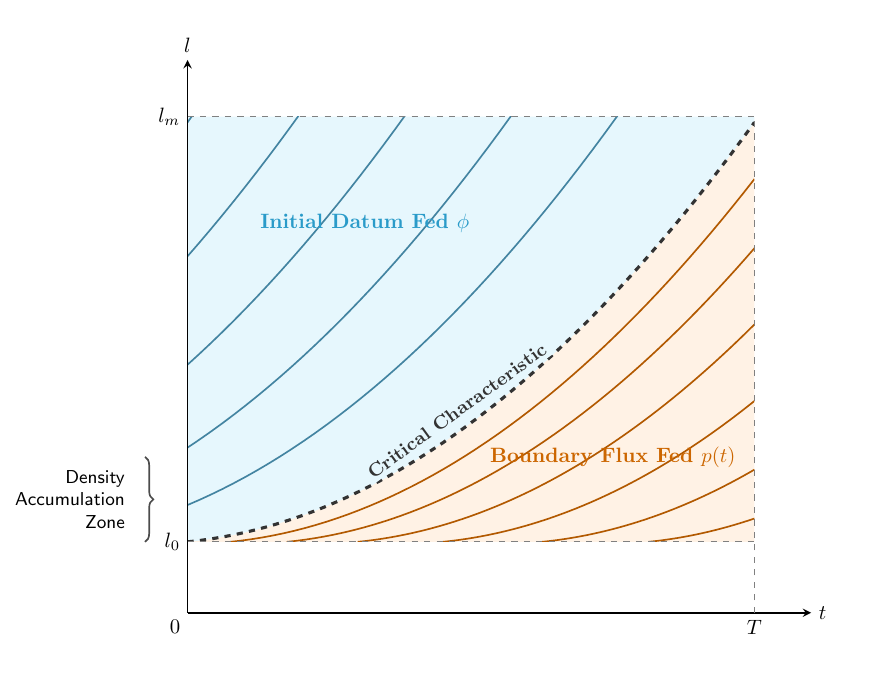}
    \caption{Typical shape of characteristic curves. The dashed critical characteristic separates the initial-datum-fed region (cyan) from the boundary-flux-fed region (orange), with its curvature determined by the growth rate $g(E, l)$.}
    \label{fig:characteristics}
\end{figure}

\subsection{Size-selective harvesting}
\label{subsec:harvesting}

The control $u=u(t,l)$ is the instantaneous fishing mortality, with the box
constraint
\[
0\le u(t,l)\le u_{\max}
\qquad
\text{for a.e.\ }(t,l)\in(0,T)\times\Om,
\]
and the admissible $L^\infty$ class is
\[
\Uad
:=
\left\{
u\in L^\infty((0,T)\times\Om):
0\le u\le u_{\max}\ \text{a.e.}
\right\}.
\]
A threshold or minimum-size policy is
\begin{equation}\label{eq:threshold_control}
u_L(t,l)
=
u_{\max}\mathbf 1_{\{l>L(t)\}},
\qquad L(t)\in[l_0,l_m].
\end{equation}
Such regulations are common because they are observable, enforceable, and protect
immature individuals \cite{ingram2006catch,wang2026breakdown}; nevertheless
\eqref{eq:threshold_control} is only one possible size-selective policy. With
$c(l)\ge0$ the value per harvested individual, the discounted yield is
\eqref{eq:intro_J}. Although $c$ is often increasing in size, monotonicity of $c$
alone does not imply a threshold policy, because the opportunity cost is the
adjoint variable, which may itself depend non-monotonically on size.

\subsection{Controlled population equation}
\label{subsec:state_equation}

The controlled dynamics are
\begin{equation}\label{eq:state}
\partial_t x
+
\partial_l\!\big(g(E(t),l)x\big)
=
-\big(\mu(E(t),l)+u(t,l)\big)x,
\qquad
(t,l)\in(0,T)\times\Om,
\end{equation}
with $E(t)=\int_{l_0}^{l_m}\chi(l)x(t,l)\dd l$ and data
\begin{equation}\label{eq:state_data}
x(0,l)=\phi(l),
\qquad
g(E(t),l_0)x(t,l_0)=p(t).
\end{equation}
Expanding the flux gives the characteristic form
\[
\partial_t x
+
g(E(t),l)\partial_lx
=
-\big(
\partial_lg(E(t),l)
+
\mu(E(t),l)
+
u(t,l)
\big)x,
\]
in which $-\partial_l g\,x$ is the Jacobian compression/dilution term of nonuniform
transport, not a demographic loss. The data satisfy
\[
\phi\in L^1_+(\Om)\cap L^\infty(\Om),
\qquad
p\in L^\infty_+(0,T),
\qquad
\chi\in L^\infty_+(\Om).
\]
Additional $BV$ assumptions are imposed only when needed in
Sections~\ref{sec:wellposed} and \ref{sec:adjoint}.

\subsection{Population balance and a priori estimate}
\label{subsec:balance}

Integrating \eqref{eq:state} over $\Om$ and using the inflow condition,
\begin{equation}\label{eq:abundance}
\begin{aligned}
N'(t)
&=
p(t)
-
g(E(t),l_m)x(t,l_m)
-
\int_{l_0}^{l_m}
\big(\mu(E(t),l)+u(t,l)\big)x(t,l)\dd l \\[4pt]
&\le
p(t)-\mu_{\min}N(t),
\end{aligned}
\end{equation}
so that
\[
N(t)
\le
e^{-\mu_{\min}t}N(0)
+
\int_0^t e^{-\mu_{\min}(t-s)}p(s)\dd s
\quad(\mu_{\min}>0),
\]
with the obvious modification when $\mu_{\min}=0$. Consequently
\begin{equation}\label{eq:Ebound}
0\le E(t)
\le
\|\chi\|_{L^\infty(\Om)}N(t)
\le M_T,
\end{equation}
for a constant $M_T$ depending on the horizon and data but not on the particular
admissible control. This confines all coefficient evaluations to the compact
environmental interval $[0,M_T]$.

The non-linear coupling of the system is illustrated in Figure~\ref{fig:feedback}. The feedback loop is closed as the population density determines the environmental index $E(t)$, which in turn modulates the vital rates.

\begin{figure}[htbp]
    \centering
    \includegraphics[width=0.85\textwidth]{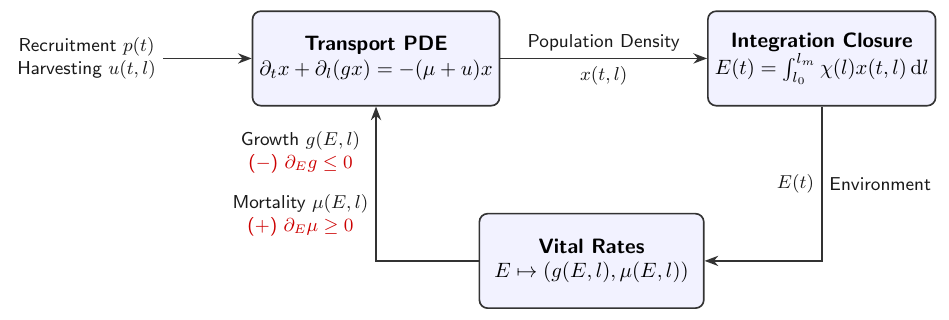}
    \caption{Feedback loop of the size-structured model. The population density couples with the integration closure via the environment-dependent vital rates ($g, \mu$).}
    \label{fig:feedback}
\end{figure}

\subsection{Exogenous recruitment and scope}
\label{subsec:recruitment_scope}

The inflow $p(t)$ is prescribed independently of the population, so the model
describes exogenous recruitment. If $p(t)>0$, the population may persist
even when its own reproductive output is insufficient to replace harvested
individuals; positivity of the forced system must not be read as demographic
self-sustainability. The endogenous counterpart replaces \eqref{eq:state_data} by
a renewal boundary condition; this diagnostic is developed in
Section~\ref{sec:persistence}. The optimization problem studied here is the forced
problem \eqref{eq:state}--\eqref{eq:state_data}.

% ============================================================================
\section{Stationary closure and failure of pointwise monotonicity}
\label{sec:stationary}

Throughout this section assume
\begin{equation}\label{eq:stationary_data}
p(t)\equiv p>0,
\qquad
u(t,l)\equiv u(l)\in L^\infty(\Om),
\qquad
0\le u(l)\le u_{\max}\ \text{a.e.},
\end{equation}
together with \eqref{eq:H1}--\eqref{eq:H5} and bounded continuous $E$-derivatives
on every compact environmental interval, justifying differentiation under the
integral sign.

\subsection{Frozen-environment stationary profile}
\label{subsec:stationary_profile}

A stationary pair $(x,E)$ satisfies
\[
\frac{\dd}{\dd l}\big(g(E,l)x(l)\big)
=
-\big(\mu(E,l)+u(l)\big)x(l),
\quad
g(E,l_0)x(l_0)=p,
\quad
E=\int_{l_0}^{l_m}\chi(l)x(l)\dd l.
\]
With the mortality-to-growth ratio
\[
q(E,l):=\frac{\mu(E,l)+u(l)}{g(E,l)}\ge0
\]
and $y_E:=g(E,\cdot)x_E$ solving $y_E'=-q(E,\cdot)y_E$, $y_E(l_0)=p$, one obtains
\begin{equation}\label{eq:stationary_profile}
x_E(l)
=
\frac{p}{g(E,l)}
\exp\!\left(
-\int_{l_0}^{l}
\frac{\mu(E,\xi)+u(\xi)}{g(E,\xi)}
\dd\xi
\right)
=
p\,\frac{1}{g(E,l)}\,\sigma_E(l),
\end{equation}
where $\sigma_E(l)=\exp(-\int_{l_0}^l q)$ is the survival to size $l$ and
$1/g(E,l)$ the residence time per unit size.

\begin{proposition}[Frozen-environment profile]
\label{prop:frozen_stationary_profile}
For every $E\in[0,M_{\mathrm{st}}]$ and every $u$ as in
\eqref{eq:stationary_data}, the stationary transport equation with flux datum $p$
has the unique solution \eqref{eq:stationary_profile}, and
$0<x_E(l)\le p/g_{\min}(M_{\mathrm{st}})$ on $\Om$.
\end{proposition}

\subsection{Scalar environmental closure}
\label{subsec:scalar_closure}

Define the closure map
\begin{equation}\label{eq:Phi}
\Phi(E):=\int_{l_0}^{l_m}\chi(l)x_E(l)\dd l,
\end{equation}
so that the stationary consistency condition is $E=\Phi(E)$.

\begin{proposition}[Existence of a stationary environment]
\label{prop:Phi_existence}
With $g_{\min}(M):=\inf_{[0,M]\times\Om}g>0$ from \eqref{eq:H2},
\begin{equation}\label{eq:Phi_bound}
0\le\Phi(E)\le \frac{p\|\chi\|_{L^\infty(\Om)}}{g_{\min}(M)}(l_m-l_0)=:G(M)
\qquad\text{for }E\in[0,M].
\end{equation}
Assume the closure-confinement condition
\begin{equation}\label{eq:Hst}
\exists\,M_{\mathrm{st}}>0:\qquad
\Phi\big([0,M_{\mathrm{st}}]\big)\subseteq[0,M_{\mathrm{st}}].
\tag{H-st}
\end{equation}
Then $\Phi$ is continuous on $[0,M_{\mathrm{st}}]$ and $E=\Phi(E)$ has at least one
fixed point $E^*\in[0,M_{\mathrm{st}}]$. If $\chi\ge0$ is positive on a set of
positive measure, then $\Phi(E)>0$ and every fixed point is positive.
\end{proposition}

\begin{remark}[On \eqref{eq:Hst} and how it is met]
\label{rmk:Hst}
Condition \eqref{eq:Hst} is a genuine restriction, not automatic. Two sufficient
conditions are useful. \emph{(a) Crude a priori bound.} If $G(M_{\mathrm{st}})\le
M_{\mathrm{st}}$ for some $M_{\mathrm{st}}$, then \eqref{eq:Hst} holds; but $G$ may
fail this for all $M$, since $g_{\min}(M)$ can decay so fast that $G$ grows faster
than $M$ (e.g.\ $g_{\min}(M)\sim(1+M)^{-2}$ gives $G(M)\sim(1+M)^2$). This crude
test is often far too conservative because \eqref{eq:Phi_bound} discards survival
and the weighting. \emph{(b) Decreasing closure map.} If $\Phi$ is nonincreasing
on $[0,\infty)$, then $M_{\mathrm{st}}:=\Phi(0)$ satisfies \eqref{eq:Hst}, since
$0\le\Phi(E)\le\Phi(0)$ for $E\in[0,\Phi(0)]$. For the calibrated von Bertalanffy
model of Section~\ref{sec:concrete} the crude test (a) does not hold at the
baseline (the bound $G$ overestimates $\Phi$ by orders of magnitude), but $\Phi$
is decreasing on the relevant range (Figure~\ref{fig:closure}), so (b) applies and
we use $M_{\mathrm{st}}=\Phi(0)\approx1.04$; this is the interval on which the
numerical closure diagnostics are evaluated. Monotonicity of $\Phi$ also yields
uniqueness (Theorem~\ref{thm:integrated_monotonicity}); where $\Phi$ is not
globally monotone, \eqref{eq:Hst} must be checked directly, as we do numerically.
\end{remark}

\begin{remark}[Relation to existing theory]
\label{rmk:stationary_novelty}
Environmental closure equations and their fixed points are classical
\cite{metz2014dynamics,diekmann2003steady,gyllenberg2007mathematical}, and stationary
size-only exploitation has been studied \cite{davydov2011optimization}. We claim no
novelty for \eqref{eq:Phi} or Proposition~\ref{prop:Phi_existence}. The point of
this section is that feedback through the rates makes the sign of $\Phi'(E)$ a
competition of two integrated effects, so monotonicity cannot be read off from
$\partial_Eg\le0$, $\partial_E\mu\ge0$ pointwise.
\end{remark}

\subsection{Sensitivity of the frozen profile}
\label{subsec:profile_sensitivity}

To avoid the symbol clash with the discounted-mortality coefficient introduced in
Section~\ref{sec:stationary_switching}, we write the growth-sensitivity coefficient
here as $\alpha$:
\begin{equation}\label{eq:a_b_definitions}
\alpha(E,l):=-\frac{\partial_Eg(E,l)}{g(E,l)},
\qquad
b(E,l):=\partial_Eq(E,l)
=\partial_E\!\left(\frac{\mu(E,l)+u(l)}{g(E,l)}\right).
\end{equation}
Under \eqref{eq:H4}, $\alpha(E,l)\ge0$ and
\begin{equation}\label{eq:b_expansion}
b(E,l)
=
\frac{\partial_E\mu(E,l)\,g(E,l)-\big(\mu(E,l)+u(l)\big)\partial_Eg(E,l)}{g(E,l)^2}
\ge0,
\end{equation}
the latter a sum of two nonnegative terms. The logarithmic derivative of
\eqref{eq:stationary_profile} is
\begin{equation}\label{eq:dE_log_x}
\partial_E\log x_E(l)
=
\alpha(E,l)-\int_{l_0}^{l}b(E,\xi)\dd\xi,
\quad
\partial_Ex_E(l)
=
x_E(l)\left[\alpha(E,l)-\int_{l_0}^{l}b(E,\xi)\dd\xi\right].
\end{equation}
Thus the sensitivity is a positive local residence-time term $\alpha$ minus a
negative cumulative survival/progression term.

\subsection{Failure of pointwise monotonicity at the inflow}
\label{subsec:pointwise_failure}

A naive sufficient condition for strict decrease of $\Phi$ is
$\partial_Ex_E(l)<0$ for all $l$, i.e.\
$\alpha(E,l)<\int_{l_0}^l b(E,\xi)\dd\xi$ for all $l$. This is incompatible with
the inflow structure.

\begin{proposition}[Pointwise monotonicity fails at the inflow]
\label{prop:pointwise_failure}
For every $E\ge0$,
\begin{equation}\label{eq:inflow_sensitivity}
x_E(l_0)=\frac{p}{g(E,l_0)},
\qquad
\partial_Ex_E(l_0)
=
-\frac{p\,\partial_Eg(E,l_0)}{g(E,l_0)^2}
=
x_E(l_0)\alpha(E,l_0)\ge0.
\end{equation}
If $\partial_Eg(E,l_0)<0$, then $\partial_Ex_E(l_0)>0$, and by continuity there
is $\delta_E>0$ with $\partial_Ex_E(l)>0$ for $l\in[l_0,l_0+\delta_E)$. Hence the
pointwise decreasing-profile condition cannot hold whenever growth is strictly
crowding-suppressed at the inflow.
\end{proposition}

\begin{remark}[Ecological reading]
\label{rmk:inflow_interpretation}
The boundary prescribes the entering flux $p$, not the entering density. Since
$x_E(l_0)=p/g(E,l_0)$, slower growth keeps recruits longer near $l_0$ and raises
their density there, before any cumulative mortality acts.
\end{remark}

\label{subsec:integrated_derivative}

Differentiating \eqref{eq:Phi} under the integral sign and using
\eqref{eq:dE_log_x},
\[
\Phi'(E)
=
\int_{l_0}^{l_m}
\chi(l)x_E(l)
\left[
\alpha(E,l)-\int_{l_0}^{l}b(E,\xi)\dd\xi
\right]\dd l.
\]
With the tail load
$\displaystyle W_E(\xi):=\int_{\xi}^{l_m}\chi(l)x_E(l)\dd l$,
Fubini's theorem gives
\begin{equation}\label{eq:Phi_prime_tail}
\Phi'(E)
=
\underbrace{\int_{l_0}^{l_m}\chi(l)x_E(l)\alpha(E,l)\dd l}_{\mathsf A(E)}
-
\underbrace{\int_{l_0}^{l_m}b(E,\xi)W_E(\xi)\dd\xi}_{\mathsf C(E)} .
\end{equation}
Here $\mathsf A(E)\ge0$ is the residence-time amplification induced by
crowding-suppressed growth, while $\mathsf C(E)\ge0$ is the cumulative
survival/progression loss, weighted by the downstream tail load $W_E$.

\begin{theorem}[Integrated monotonicity: a sufficient uniqueness certificate]
\label{thm:integrated_monotonicity}
If
\[
\mathsf C(E)>\mathsf A(E)
\qquad\text{for every }E\in[0,M_{\mathrm{st}}],
\]
equivalently $\Phi'(E)<0$ on $[0,M_{\mathrm{st}}]$, then $\Phi$ is strictly
decreasing on $[0,M_{\mathrm{st}}]$, and $E=\Phi(E)$ has exactly one solution
$E^*\in[0,M_{\mathrm{st}}]$, with unique profile $x^*=x_{E^*}$. This global
integrated-monotonicity condition is sufficient for uniqueness; it is not
necessary, since a non-monotone $\Phi$ may still meet the identity once.
\end{theorem}

\begin{corollary}[Uniform margin]
\label{cor:uniform_integrated_margin}
If $\mathsf C(E)-\mathsf A(E)\ge\kappa>0$ on $[0,M_{\mathrm{st}}]$, then
$\Phi'(E)\le-\kappa$, so the sufficient condition of
Theorem~\ref{thm:integrated_monotonicity} holds with a margin and is robust under
perturbations preserving it.
\end{corollary}

\begin{remark}[Why the integrated form is the appropriate test]
\label{rmk:integrated_weaker}
Pointwise decrease would imply $\Phi'<0$, but
Proposition~\ref{prop:pointwise_failure} shows it is unattainable at the inflow.
The integrated condition \eqref{eq:Phi_prime_tail} permits
$\partial_Ex_E(l)>0$ near $l_0$ provided the weighted negative sensitivity at
larger sizes dominates. It is thus not merely weaker than a pointwise test; it is
the monotonicity-based sufficient uniqueness condition compatible with the model's
inflow structure.
\end{remark}

\begin{remark}[Forward pointer]
\label{rmk:forward_pointer_unification}
The quantity $\Phi'(E^*)=\mathsf A-\mathsf C$ reappears in
Section~\ref{sec:stationary_switching} as the zero-discount value of the adjoint
feedback gain: Theorem~\ref{thm:closure_gain_identity} shows
$B|_{r=0}=\Phi'(E^*)$. The residence-time term $\mathsf A$ is, there, exactly
the boundary leftover of the auxiliary backward equation, and $\mathcal C$ its
forward/adjoint dual.
\end{remark}

\subsection{A computable closure certificate}
\label{subsec:closure_certificate}

For $E\in[0,M_{\mathrm{st}}]$, define
\[
\Delta_\Phi(E):=\mathsf C(E)-\mathsf A(E)=-\Phi'(E).
\]
Then $\Delta_\Phi(E)>0$ is equivalent to $\Phi'(E)<0$.
A numerical check of
$\Delta_\Phi>0$ on a grid is evidence for strict monotonicity; a rigorous
certificate follows by enclosing the profile, quadratures, and interpolation
error with interval arithmetic, or by combining grid values with an a posteriori
bound on $\sup_{[E_j,E_{j+1}]}|\Phi''|$. The von Bertalanffy verification appears
in Section~\ref{sec:concrete}.

% ============================================================================
\section{Finite-horizon well-posedness}
\label{sec:wellposed}

Existence, stability, and compactness for structured transport equations with
nonlocal coefficients are available in measure-valued and Lagrangian frameworks
\cite{gwiazda2010structured,carrillo2012structured,
dull2021spaces}. Because the present model has a strictly
positive $C^1$ speed and bounded inflow, a direct characteristic argument
suffices. We work in the spatial-$BV$-compatible class
\[
\mathcal U_T^{\mathrm{wp}}(C_u)
:=
\Big\{
u\in L^\infty((0,T)\times\Om):
0\le u\le u_{\max}\ \text{a.e.},\
\operatorname*{ess\,sup}_{t\in(0,T)}\TV_\Om\big(u(t,\cdot)\big)\le C_u
\Big\},
\]
which contains the threshold controls and coincides with the admissible class of
Section~\ref{sec:control_existence}. Assume
\begin{equation}\label{eq:wp_data}
\phi\in BV(\Om)\cap L^\infty_+(\Om),
\qquad
p\in BV(0,T)\cap L^\infty_+(0,T),
\qquad
\chi\in W^{1,\infty}_+(\Om),
\end{equation}
together with \eqref{eq:H1}--\eqref{eq:H5} and, on every $[0,M]$,
\begin{equation}\label{eq:wp_coefficients}
\sup_{E\in[0,M]}
\Big(
\|\partial_Eg(E,\cdot)\|_{W^{1,\infty}(\Om)}
+\|\partial_E\mu(E,\cdot)\|_{L^\infty(\Om)}
+\|\partial_lg(E,\cdot)\|_{BV(\Om)}
+\|\mu(E,\cdot)\|_{BV(\Om)}
\Big)<\infty.
\end{equation}

\subsection{Weak formulation and population bound}
\label{subsec:weak_solution}

\begin{definition}[Weak solution]
\label{def:weak}
For $u\in\mathcal U_T^{\mathrm{wp}}(C_u)$, a nonnegative
$x\in C([0,T];L^1(\Om))\cap L^\infty((0,T)\times\Om)$ with
$E(t)=\int_\Om\chi x(t,\cdot)$ is a weak solution if, for all
$\psi\in C^1([0,T]\times\Om)$ with $\psi(T,\cdot)=0$ and $\psi(\cdot,l_m)=0$,
\[
\int_0^T\!\!\int_\Om
x\big[\partial_t\psi+g(E,l)\partial_l\psi-(\mu(E,l)+u)\psi\big]\dd l\dd t
+\int_\Om\phi\,\psi(0,\cdot)\dd l+\int_0^T p\,\psi(\cdot,l_0)\dd t=0.
\]
\end{definition}

The population balance \eqref{eq:abundance} gives $N'\le p-\mu_{\min}N$, hence
$N(t)\le C_T^N$ and $0\le E(t)\le\|\chi\|_{L^\infty(\Om)}C_T^N=:M_T$, uniformly
over $\mathcal U_T^{\mathrm{wp}}(C_u)$.

For the time-modulus estimates of Section~\ref{sec:control_existence} and for the
environmental-regularity bound below we record a Green identity for the
characteristic solution. It is not a consequence of
Definition~\ref{def:weak} alone, whose test functions vanish at $l_m$; it is
proved directly from the representation \eqref{eq:frozen_representation}.

\begin{lemma}[Outflow trace and Green identity]
\label{lem:green_trace}
Let $x$ be the (frozen or coupled) characteristic solution with data
\eqref{eq:wp_data}. Then the outflow flux trace
$\beta(t):=g(E(t),l_m)\,x(t,l_m^-)$ is well defined for a.e.\ $t$ and lies in
$L^\infty(0,T)$, with $\|\beta\|_{L^\infty}\le g_{\max}C_T^0$, $C_T^0$ the
$L^\infty$ state bound \eqref{eq:frozen_uniform_bounds}. Moreover, for every
$\psi\in W^{1,\infty}(\Om)$ and $0\le t\le t+h\le T$,
\begin{equation}\label{eq:green_time_identity}
\begin{aligned}
\int_\Om\big(x(t+h,\cdot)-x(t,\cdot)\big)\psi\dd l
={}&\int_t^{t+h}\!\!\int_\Om\Big(g(E,l)x\,\partial_l\psi-(\mu(E,l)+u)x\,\psi\Big)\dd l\dd s\\
&+\int_t^{t+h}p(s)\psi(l_0)\dd s
-\int_t^{t+h}g(E(s),l_m)x(s,l_m^-)\psi(l_m)\dd s.
\end{aligned}
\end{equation}
\end{lemma}

\subsection{Frozen-environment problem}
\label{subsec:frozen_environment}

Let $\mathcal B_{M_T}:=\{E\in C([0,T]):0\le E\le M_T\}$ and, for $L_E>0$, the
closed Lipschitz subclass
\[
\mathcal B_{M_T}^{L_E}
:=\big\{E\in\mathcal B_{M_T}:|E(t)-E(s)|\le L_E|t-s|\ \forall\, s,t\big\},
\]
which is convex and closed in $C([0,T])$ (a uniform limit of $L_E$-Lipschitz
functions is $L_E$-Lipschitz). For $E\in\mathcal B_{M_T}$ let
$\eta_E(s;t,l)$ solve
$\frac{\dd}{\dd s}\eta_E=g(E(s),\eta_E)$, $\eta_E(t;t,l)=l$. Since $g\ge g_{\min}$,
each backward characteristic meets the initial line $s=0$ or the inflow $l=l_0$
(entrance time $\tau_E(t,l)$). With
\[
\mathcal A_E(a,t;l)
:=
\exp\!\Big(
-\!\int_a^t\!\big[\partial_lg+\mu+u\big](E(s),\eta_E(s;t,l))\dd s
\Big),
\]
the representation is
\begin{equation}\label{eq:frozen_representation}
x_E(t,l)
=
\begin{cases}
\phi\big(\eta_E(0;t,l)\big)\mathcal A_E(0,t;l),
& \eta_E(0;t,l)>l_0,\\[2mm]
\dfrac{p(\tau_E(t,l))}{g(E(\tau_E(t,l)),l_0)}\,\mathcal A_E(\tau_E(t,l),t;l),
& \eta_E(0;t,l)\le l_0.
\end{cases}
\end{equation}

\begin{lemma}[Frozen-environment solution: existence and amplitude bounds]
\label{lem:frozen}
For every $E\in\mathcal B_{M_T}$, $u\in\mathcal U_T^{\mathrm{wp}}(C_u)$, the
frozen problem has a unique nonnegative weak solution
$x_E\in C([0,T];L^1(\Om))\cap L^\infty((0,T)\times\Om)$ with
\begin{equation}\label{eq:frozen_uniform_bounds}
\sup_{t\in[0,T]}
\big(\|x_E(t,\cdot)\|_{L^1}+\|x_E(t,\cdot)\|_{L^\infty}\big)
+\|g(E,l_m)x_E(\cdot,l_m^-)\|_{L^\infty(0,T)}
\le C_T^0,
\end{equation}
$C_T^0$ independent of $E\in\mathcal B_{M_T}$ (in particular, independent of any
modulus of continuity of $E$).
\end{lemma}

\begin{remark}[Why a $BV$ bound needs more than $E\in C$]
\label{rmk:bv_needs_lipschitz}
The boundary-fed branch of \eqref{eq:frozen_representation} contains the factor
$p(\tau_E(t,l))/g(E(\tau_E(t,l)),l_0)$. With $p\in BV(0,T)$ but $E$ merely
continuous, the map $t\mapsto1/g(E(t),l_0)$ can have infinite total variation, so
its composition with the (monotone) entrance-time map need not be of bounded
spatial variation. A spatial-$BV$ bound for $x_E$ therefore cannot hold uniformly
over $\mathcal B_{M_T}$; it requires the environmental path itself to be of
bounded temporal variation. We obtain this through a uniform Lipschitz bound on
the coupled environment, derived next, and then restrict the fixed point
to $\mathcal B_{M_T}^{L_E}$.
\end{remark}

\begin{lemma}[Environmental regularity: uniform Lipschitz bound]
\label{lem:env_lipschitz}
Assume $\chi\in W^{1,\infty}_+(\Om)$. For every $E\in\mathcal B_{M_T}$ the map
$t\mapsto(\mathcal FE)(t):=\int_\Om\chi(l)x_E(t,l)\dd l$ is Lipschitz, with
\begin{equation}\label{eq:E_prime_identity}
(\mathcal FE)'(t)
=\chi(l_0)p(t)
-\chi(l_m)g(E(t),l_m)x_E(t,l_m^-)
+\int_\Om g(E,l)x_E\,\chi'(l)\dd l
-\int_\Om\chi(l)\big(\mu(E,l)+u\big)x_E\dd l
\end{equation}
for a.e.\ $t$, and the explicit bound
\begin{equation}\label{eq:E_prime_bound}
\operatorname*{ess\,sup}_{t\in(0,T)}|(\mathcal FE)'(t)|
\le L_E^\star
:=\|\chi\|_\infty\|p\|_\infty
+\|\chi\|_\infty g_{\max}C_T^0
+\big(g_{\max}\|\chi'\|_\infty+\|\chi\|_\infty(\mu_{\max,T}+u_{\max})\big)C_T^0,
\end{equation}
with $C_T^0$ from \eqref{eq:frozen_uniform_bounds} and
$\mu_{\max,T}:=\sup_{(E,l)\in[0,M_T]\times\Om}\mu(E,l)<\infty$ (finite by
\eqref{eq:H1} on the compact set). The constant $L_E^\star$ is
independent of $E\in\mathcal B_{M_T}$. In particular $\mathcal F$ maps
$\mathcal B_{M_T}$ into $\mathcal B_{M_T}^{L_E^\star}$.
\end{lemma}

We henceforth fix $L_E:=L_E^\star$ and work in $\mathcal B_{M_T}^{L_E}$.

\begin{lemma}[Frozen-environment spatial $BV$ bound under a Lipschitz environment]
\label{lem:frozen_bv}
For every $E\in\mathcal B_{M_T}^{L_E}$ and $u\in\mathcal U_T^{\mathrm{wp}}(C_u)$,
\begin{equation}\label{eq:frozen_bv_bound}
\sup_{t\in[0,T]}\TV_\Om\big(x_E(t,\cdot)\big)\le C_T,
\end{equation}
with $C_T$ depending on the data, the coefficient bounds \eqref{eq:wp_coefficients},
$C_u$, and $L_E$, but not otherwise on $E$.
\end{lemma}

\subsection{Stability in the environmental trajectory}
\label{subsec:environmental_stability}

\begin{lemma}[Short-time environmental stability]
\label{lem:short_stability}
Let $E_1,E_2\in\mathcal B_{M_T}^{L_E}$ with frozen solutions $x_{E_1},x_{E_2}$ for
the same data and control. There is $C_T>0$ independent of $E_1,E_2$ with, for
every $\tau\in(0,T]$,
\begin{equation}\label{eq:short_stability}
\sup_{t\in[0,\tau]}\|x_{E_1}(t,\cdot)-x_{E_2}(t,\cdot)\|_{L^1(\Om)}
\le C_T\tau\|E_1-E_2\|_{C([0,\tau])}.
\end{equation}
\end{lemma}

\subsection{Nonlinear fixed point}
\label{subsec:nonlinear_fixed_point}

By Lemma~\ref{lem:env_lipschitz}, $\mathcal F$ maps $\mathcal B_{M_T}$ into the
closed convex Lipschitz class $\mathcal B_{M_T}^{L_E}$; in particular
$\mathcal F:\mathcal B_{M_T}^{L_E}\to\mathcal B_{M_T}^{L_E}$. On
$\mathcal B_{M_T}^{L_E}$ the frozen states obey the uniform spatial-$BV$ bound of
Lemma~\ref{lem:frozen_bv}, so Lemma~\ref{lem:short_stability} applies and gives
\[
\|\mathcal FE_1-\mathcal FE_2\|_{C([0,\tau])}
\le\|\chi\|_{L^\infty}C_T\tau\|E_1-E_2\|_{C([0,\tau])},
\]
a contraction on $\mathcal B_{M_T}^{L_E}$ for small $\tau$.

\begin{theorem}[Finite-horizon well-posedness]
\label{thm:wellposed}
Assume \eqref{eq:H1}--\eqref{eq:H5}, \eqref{eq:wp_data} (in particular
$\chi\in W^{1,\infty}_+(\Om)$), and \eqref{eq:wp_coefficients}. For every
$u\in\mathcal U_T^{\mathrm{wp}}(C_u)$, problem
\eqref{eq:state}--\eqref{eq:state_data} has a unique nonnegative weak solution
$x^u\in C([0,T];L^1(\Om))\cap L^\infty((0,T)\times\Om)$ whose environment
$E^u(t)=\int_\Om\chi x^u(t,\cdot)$ is $L_E$-Lipschitz, with
\[
\sup_{t\in[0,T]}
\big(\|x^u(t,\cdot)\|_{L^1}+\|x^u(t,\cdot)\|_{L^\infty}+\TV_\Om(x^u(t,\cdot))\big)
\le C_T,
\]
$C_T$ depending on $T$, data, coefficient bounds, and $C_u$, not on the control.
\end{theorem}

\begin{proposition}[Continuous dependence]
\label{prop:continuous_dependence}
For data and controls $(\phi_i,p_i,u_i)$, $u_i\in\mathcal U_T^{\mathrm{wp}}(C_u)$,
\begin{equation}\label{eq:continuous_dependence}
\sup_{t}\|x_1-x_2\|_{L^1(\Om)}
\le
C_T\big(\|\phi_1-\phi_2\|_{L^1}+\|p_1-p_2\|_{L^1(0,T)}+\|u_1-u_2\|_{L^1}\big),
\end{equation}
with $C_T$ uniform under common $L^\infty$ and spatial-$BV$ bounds.
\end{proposition}

\begin{remark}[Scope]
\label{rmk:wp_scope}
The spatial-$BV$ hypothesis controls compositions of $u$ and the data with nearby
flows; an $L^\infty$ bound alone controls amplitude but not sensitivity to
spatial translation. We therefore do not claim strong continuous dependence
uniformly over $\Uad$. Broader (Lagrangian and measure-valued)
formulations are possible \cite{gwiazda2010structured,yu2026beyond,carrillo2012structured} but are outside our scope.
\end{remark}

% ============================================================================
\section{Forced persistence and intrinsic replacement}
\label{sec:persistence}

The forced model has exogenous recruitment, so persistence under prescribed $p$
is not demographic self-sustainability. To diagnose intrinsic replacement we
introduce the endogenous renewal counterpart, used here only as a diagnostic; the
optimization remains the forced problem.

\subsection{Replacement functional and basic reproduction number}
\label{subsec:renewal}

Replace the inflow datum in \eqref{eq:state_data} by the size-specific fertility
renewal condition
\begin{equation}\label{eq:renewal_bc}
g(E(t),l_0)x(t,l_0)=\int_{l_0}^{l_m}m(l)x(t,l)\dd l,
\qquad m\in L^\infty_+(\Om).
\end{equation}
At a stationary environment $E$ with frozen profile normalized to unit recruitment
flux, $\hat x_E(l):=x_E(l)/p$, define the environment- and harvest-dependent
replacement functional
\[
\mathsf A(E;u):=\int_{l_0}^{l_m}m(l)\hat x_E(l)\dd l
=
\int_{l_0}^{l_m}\frac{m(l)}{g(E,l)}\exp\!\Big(-\!\int_{l_0}^l\frac{\mu(E,\xi)+u(\xi)}{g(E,\xi)}\dd\xi\Big)\dd l,
\]
the expected lifetime offspring of a recruit at environment $E$ under harvesting
$u$. The conventional basic reproduction number is the value at the
extinction (invasion) environment,
$\displaystyle \mathcal R_0(u):=\mathcal R(0;u)$,
the expected lifetime offspring in an empty population, as obtained from the
next-generation operator of structured-population theory
\cite{diekmann2003steady,cushing1994structured}; we reserve the symbol $\mathcal R_0$
for this invasion quantity and do not call $\mathcal R(E;u)$ a reproduction number.

\begin{proposition}[Stationary renewal state: generation balance and amplitude]
\label{prop:R0_monotone}
A nontrivial stationary renewal state has the form $x=b\,\hat x_E$ with $b>0$, and
exists at environment $E^{\dagger}>0$ if and only if both the generation-balance
condition and the environmental closure hold:
\begin{equation}\label{eq:renewal_equilibrium}
\mathcal R(E^{\dagger};u)=1
\qquad\text{and}\qquad
b=\frac{E^{\dagger}}{\displaystyle\int_{l_0}^{l_m}\chi(l)\hat x_{E^{\dagger}}(l)\dd l}.
\end{equation}
If $\chi>0$ on a set of positive measure, no positive stationary state exists at
$E=0$. Moreover, under \eqref{eq:H4}, increasing $u$ pointwise does not increase
$\mathcal R(\cdot;u)$; and if $\mathcal R(\cdot;u)$ is strictly decreasing in $E$
on $[0,M]$ with $\mathcal R(0;u)>1>\mathcal R(M;u)$, then $E^{\dagger}$ is unique.
\end{proposition}

\subsection{Equilibrium-adjusted spawning-potential ratio}
\label{subsec:SPR}

For a reference unfished equilibrium environment $E_0$ (set $u\equiv0$) and a
fished equilibrium $E^*$, define the equilibrium-adjusted spawning-potential
ratio
\[
\mathrm{SPR}_{\mathrm{eq}}
:=
\frac{\displaystyle\int_{l_0}^{l_m}m(l)\hat x_{E^*}(l)\dd l}
{\displaystyle\int_{l_0}^{l_m}m(l)\hat x_{E_0}^{\,0}(l)\dd l}
=\frac{\mathcal R(E^*;u)}{\mathcal R(E_0;0)},
\]
where $\hat x_{E_0}^{\,0}$ is the unit-recruitment unfished profile. This is a
defensible per-recruit ratio, but it compares the fished and unfished states at
their respective equilibrium environments and is therefore not the classical
fixed-vital-rate SPR: under compensatory density feedback the vital rates differ
between numerator and denominator, so $\mathrm{SPR}_{\mathrm{eq}}$ may behave
differently from a conventional fixed-environment SPR (for instance, density
release at low abundance can raise per-recruit output and partially offset
harvesting). It is the quantity against which the forced-optimal harvesting
policies of Sections~\ref{sec:adjoint}--\ref{sec:stationary_switching} are compared
in Section~\ref{sec:concrete}. We emphasize that
$\mathrm{SPR}_{\mathrm{eq}}<1$ is compatible with forced persistence: prescribed
recruitment can sustain a population whose intrinsic replacement is below unity.

% ============================================================================
\section{Compact admissible class and optimal-control existence}
\label{sec:control_existence}

The finite-horizon problem is
\[
V_T:=\sup_{u\in\Uad}J_T(u),
\qquad
J_T(u)=\int_0^T e^{-rt}\int_\Om c(l)u(t,l)x^u(t,l)\dd l\dd t.
\]
The box class $\Uad$ is weak-$\ast$ compact, but weak-$\ast$ convergence
$u_n\overset{\ast}{\rightharpoonup}u$ need not pass to the bilinear product
$u_nx^{u_n}$, the obstruction being spatial oscillation. We therefore introduce a
spatial-$BV$ restriction as an explicit compactification, designed to yield strong
compactness of the states, not of the controls. Throughout assume the
hypotheses of Theorem~\ref{thm:wellposed} and $c\in L^\infty_+(\Om)$.

\subsection{Spatial-$BV$ compactification}
\label{subsec:BV_compactification}

Fix $C_u\ge u_{\max}$ and set
\[
\UadBV:=\Big\{u\in\Uad:
\operatorname*{ess\,sup}_{t\in(0,T)}\TV_\Om\big(u(t,\cdot)\big)\le C_u\Big\},
\qquad
V_T^{BV}:=\sup_{u\in\UadBV}J_T(u).
\]
Every threshold control $u_L$ satisfies $\TV_\Om(u_L(t,\cdot))\le u_{\max}$, so all
threshold policies lie in $\UadBV$; a bang--bang control with at most $k$ spatial
switches has variation at most $ku_{\max}$. The budget admits the distributional
form
\begin{equation}\label{eq:TV_distributional}
\Big|\int_0^T\!\!\int_\Om u\,\partial_l\varphi\dd l\dd t\Big|
\le C_u\int_0^T\|\varphi(t,\cdot)\|_{L^\infty(\Om)}\dd t,
\qquad \varphi\in C_c^1((0,T)\times\Om).
\end{equation}

\begin{lemma}[Weak-$\ast$ compactness of $\UadBV$]
\label{lem:UadBV_compact}
$\UadBV$ is convex and sequentially weak-$\ast$ compact in
$L^\infty((0,T)\times\Om)$.
\end{lemma}

\begin{remark}[What is compactified]
\label{rmk:what_compactified}
$\Uad$ is already weak-$\ast$ compact; the role of $\UadBV$ is to give uniform
spatial regularity of the states and thereby make the control-to-state map
sequentially closed under weak-$\ast$ convergence of controls.
\end{remark}

\subsection{Uniform state compactness}
\label{subsec:state_compactness}

By Theorem~\ref{thm:wellposed},
\begin{equation}\label{eq:uniform_state_BV}
\sup_{u\in\UadBV}\sup_{t\in[0,T]}
\big(\|x^u(t,\cdot)\|_{L^1}+\|x^u(t,\cdot)\|_{L^\infty}+\TV_\Om(x^u(t,\cdot))\big)
\le C_T.
\end{equation}
We use the negative norm dual to $W^{1,\infty}$, defined explicitly by
\[
\|f\|_{\mathcal W^{-1,1}(\Om)}
:=\sup\Big\{\,|\langle f,\psi\rangle|:\psi\in W^{1,\infty}(\Om),\
\|\psi\|_{W^{1,\infty}(\Om)}\le1\,\Big\}.
\]

\begin{lemma}[Uniform weak time modulus]
\label{lem:weak_time_modulus}
There is $C_T>0$, independent of $u\in\UadBV$, with
\begin{equation}\label{eq:weak_time_estimate}
\|x^u(t+h,\cdot)-x^u(t,\cdot)\|_{\mathcal W^{-1,1}(\Om)}\le C_T h,
\qquad 0\le t\le t+h\le T.
\end{equation}
\end{lemma}

Using the one-dimensional interpolation inequality
\begin{equation}\label{eq:BV_interpolation}
\|v\|_{L^1(\Om)}
\le
C_\Om\|v\|_{\mathcal W^{-1,1}(\Om)}^{1/2}
\big(\|v\|_{L^1(\Om)}+\TV_\Om(v)\big)^{1/2},
\qquad v\in BV(\Om),
\end{equation}
(obtained by extending $v$ to $\R$, mollifying at scale $\varepsilon$, and
optimizing $\|v\|_{L^1}\le C_\Om(\varepsilon\TV(v)+\varepsilon^{-1}\|v\|_{\mathcal W^{-1,1}})$),
we upgrade the time modulus.

\begin{lemma}[Uniform strong time modulus]
\label{lem:strong_time_modulus}
There is $C_T>0$, independent of $u\in\UadBV$, with
$\|x^u(t+h,\cdot)-x^u(t,\cdot)\|_{L^1(\Om)}\le C_T h^{1/2}$ for
$0\le t\le t+h\le T$.
\end{lemma}

\begin{proposition}[Strong state compactness]
\label{prop:strong_state_compactness}
The set $\mathcal X_T^{BV}:=\{x^u:u\in\UadBV\}$ is relatively compact in
$C([0,T];L^1(\Om))$; in particular every $(u_n)\subset\UadBV$ has a subsequence
with $x^{u_n}\to\bar x$ strongly in $C([0,T];L^1(\Om))$.
\end{proposition}

\begin{remark}[Controls weak, states strong]
\label{rmk:weak_control_strong_state}
The compactification produces
$u_n\overset{\ast}{\rightharpoonup}u$ in $L^\infty$ together with
$x^{u_n}\to x^u$ strongly in $C([0,T];L^1)$. No time-compactness or strong
convergence of $u_n$ is claimed; the strong state convergence is what permits
passage to $u_nx^{u_n}$.
\end{remark}

\subsection{Sequential closedness}
\label{subsec:closed_control_state}

\begin{proposition}[Weak-$\ast$/strong closedness]
\label{prop:control_state_closed}
If $u_n\in\UadBV$ and $u_n\overset{\ast}{\rightharpoonup}u$ in
$L^\infty((0,T)\times\Om)$, then, after extraction, $x^{u_n}\to x^u$ strongly in
$C([0,T];L^1(\Om))$, where $x^u$ is the state of the limit $u\in\UadBV$.
\end{proposition}

\begin{proof}
By Proposition~\ref{prop:strong_state_compactness}, $x^{u_n}\to\bar x$ strongly,
hence
\[
E^{u_n}(t)=\int_\Om\chi x^{u_n}(t,\cdot)
\longrightarrow
\bar E(t):=\int_\Om\chi\bar x(t,\cdot)
\]
uniformly on $[0,T]$, so $g(E^{u_n},\cdot)\to g(\bar E,\cdot)$ and
$\mu(E^{u_n},\cdot)\to\mu(\bar E,\cdot)$ uniformly. For the bilinear term, write,
for bounded $\psi$,
\[
\int_0^T\!\!\int_\Om u_n x^{u_n}\psi
=
\int_0^T\!\!\int_\Om u_n\bar x\,\psi
+\int_0^T\!\!\int_\Om u_n(x^{u_n}-\bar x)\psi.
\]
Since $\bar x\psi\in L^1$, weak-$\ast$ convergence handles the first term; the
second is bounded by $u_{\max}\|\psi\|_\infty\|x^{u_n}-\bar x\|_{L^1}\to0$. Passing
to the limit in the weak formulation shows $\bar x$ solves the state equation with
control $u$, so $\bar x=x^u$ by uniqueness; the limit being independent of the
subsequence gives the claim.
\end{proof}

\subsection{Existence of an optimal control}
\label{subsec:optimal_existence}

\begin{theorem}[Existence in $\UadBV$]
\label{thm:optimal_existence_BV}
Under the hypotheses of Theorem~\ref{thm:wellposed} with $c\in L^\infty_+(\Om)$,
there is $u_T^*\in\UadBV$ with
$J_T(u_T^*)=V_T^{BV}=\sup_{u\in\UadBV}J_T(u)$.
\end{theorem}

\begin{proof}
Let $(u_n)\subset\UadBV$ be maximizing. By Lemma~\ref{lem:UadBV_compact},
$u_n\overset{\ast}{\rightharpoonup}u_T^*\in\UadBV$; by
Proposition~\ref{prop:control_state_closed}, $x^{u_n}\to x^{u_T^*}$ strongly in
$C([0,T];L^1)$. Writing
\[
\int_0^T\!\!\int_\Om e^{-rt}c\,u_n x^{u_n}
=\int_0^T\!\!\int_\Om e^{-rt}c\,u_n x^{u_T^*}
+\int_0^T\!\!\int_\Om e^{-rt}c\,u_n(x^{u_n}-x^{u_T^*}),
\]
the first term converges by weak-$\ast$ convergence (since
$e^{-rt}c\,x^{u_T^*}\in L^1$) and the second to zero by strong $L^1$ convergence,
so $J_T(u_n)\to J_T(u_T^*)=V_T^{BV}$.
\end{proof}

\subsection{Scope of the compactification}
\label{subsec:compactification_scope}

Theorem~\ref{thm:optimal_existence_BV} establishes existence for the restricted
problem only; we do not prove $V_T^{BV}=V_T$ nor that an unrestricted optimizer
lies in $\UadBV$. The relation is
$\displaystyle V_T^{BV}\le V_T$,
with equality not established. We accordingly read all first-order results below
as an optimality theory for $\UadBV$. In particular, the ecological conclusions of
Section~\ref{sec:concrete} should be read as ``among policies of bounded spatial
complexity, an optimum exists and (conditionally) is minimum-size,'' not as a
statement about the unrestricted optimum.

\begin{remark}[Alternative relaxed formulations]
\label{rmk:relaxed_formulations}
Measure-valued harvesting controls provide an alternative route when optimal
sequences develop concentrations or fine oscillations
\cite{hritonenko2023existence,coclite2017time,wang2025multi}; developing that
relaxation and deciding equality with $V_T^{BV}$ is beyond our scope.
\end{remark}

\subsection{Threshold policies}
\label{subsec:threshold_compactification}

For minimum-size controls, compactness can be imposed on the threshold itself:
\[
\mathcal L_{ad}:=\big\{L\in BV(0,T):l_0\le L(t)\le l_m\ \text{a.e.},\
\TV_{(0,T)}(L)\le C_L\big\},
\]
with $\|u_{L_1}-u_{L_2}\|_{L^1((0,T)\times\Om)}=u_{\max}\int_0^T|L_1-L_2|$.

\begin{proposition}[Existence in the threshold class]
\label{prop:threshold_existence}
There is $L_T^*\in\mathcal L_{ad}$ with
$J_T(u_{L_T^*})=\sup_{L\in\mathcal L_{ad}}J_T(u_L)$.
\end{proposition}

\begin{remark}[Stationary threshold]
\label{rmk:stationary_threshold_existence}
For a stationary threshold $L\in[l_0,l_m]$ no temporal compactification is needed:
$[l_0,l_m]$ is compact and $L\mapsto J_{\mathrm{st}}(L)$ continuous, so a maximizer
exists. Section~\ref{sec:stationary_switching} determines when it is consistent
with the full nonlocal switching condition.
\end{remark}

% ============================================================================
\section{Adjoint and bang--bang variational inequality}
\label{sec:adjoint}

We derive first-order conditions for $\sup_{u\in\UadBV}J_T(u)$. Maximum principles
and bang--bang laws are well established for related models
\cite{brokate1985pontryagin,kato2008optimal,yu2026from,hritonenko2009bang,anita2019optimal,gao2022rolling}; the distinctive feature is the nonlocal term created when
the environment enters the rates. The correct general condition is a variational
inequality over the tangent cone of $\UadBV$; a pointwise bang--bang rule follows
only when the spatial-variation constraint is inactive.

We emphasize at the outset that this finite-horizon adjoint theory is  conditional. The nonlocal functional
\eqref{eq:Gamma_BV} pairs a measure $D_lx^*$ with the costate, and the
single-crossing analysis treats $S^*(t,\cdot)$ as continuous in size; both require
$\lambda^*(t,\cdot)\in C(\Om)$, which---as Remark~\ref{rmk:costate_discontinuity}
shows---can fail under the spatial-$BV$ control class. We therefore carry an
explicit size-continuity hypothesis \eqref{eq:costate_regularity} throughout this
section. The stationary theory of Section~\ref{sec:stationary_switching} is free
of this issue: its adjoint is an ODE with bounded coefficients and is absolutely
continuous, hence continuous in size, even for bang--bang $u(l)$.

Let $u^*\in\UadBV$ be optimal (Theorem~\ref{thm:optimal_existence_BV}), with
$x^*:=x^{u^*}$ and $E^*(t):=\int_\Om\chi x^*(t,\cdot)$. Following
\cite{diekmann2025age}, who note that solution operators of size-structured
consumer--resource models may fail to be differentiable in general, we impose the
clean regularity that the rates depend smoothly on $E$ as maps into the
spatial spaces dictated by the characteristic argument:
\begin{equation}\label{eq:adjoint_assumptions}
c\in C^1(\Om),
\quad
\chi\in C(\Om),
\quad
E\mapsto g(E,\cdot)\in C^1\big(\R_+;W^{1,\infty}(\Om)\big),
\quad
E\mapsto\mu(E,\cdot)\in C^1\big(\R_+;L^\infty(\Om)\big),
\end{equation}
and $x^*\in L^\infty(0,T;BV(\Om))\cap L^\infty((0,T)\times\Om)$. The
$C^1(\R_+;W^{1,\infty}(\Om))$ hypothesis on $g$ supplies not only the zeroth-order
expansion of $g$ but the first-order expansion of $\partial_lg$ in $E$ needed by
the characteristic attenuation, replacing a list of separate mixed-derivative
bounds. We additionally assume the optimal-state costate regularity
\begin{equation}\label{eq:costate_regularity}
\lambda^*\in C\big([0,T];C(\Om)\big),
\tag{R}
\end{equation}
discussed and justified in Remark~\ref{rmk:costate_discontinuity}.

\subsection{Linearized state equation}
\label{subsec:linearized_state}

For a feasible bounded direction $h$ of finite spatial variation, set
$u^\varepsilon=u^*+\varepsilon h$, $x^\varepsilon=x^{u^\varepsilon}$,
$E^\varepsilon=\int_\Om\chi x^\varepsilon$, and let
$z=\lim_{\varepsilon\downarrow0}(x^\varepsilon-x^*)/\varepsilon$,
$\delta E(t)=\int_\Om\chi z(t,\cdot)$. Formal differentiation gives
\begin{equation}\label{eq:linearized_state}
\partial_t z+\partial_l\!\big(g(E^*,l)z\big)
+\partial_l\!\big(\partial_Eg(E^*,l)x^*\delta E\big)
=
-\big(\mu(E^*,l)+u^*\big)z-\partial_E\mu(E^*,l)x^*\delta E-hx^*,
\end{equation}
with $z(0,\cdot)=0$ and, since the inflow flux is control-independent,
\begin{equation}\label{eq:linearized_boundary}
g(E^*,l_0)z(t,l_0)+\partial_Eg(E^*,l_0)x^*(t,l_0^+)\delta E(t)=0.
\end{equation}

\begin{proposition}[Directional differentiability]
\label{prop:state_differentiability}
For every bounded feasible direction $h$ of finite spatial variation, the
difference quotient $(x^{u^*+\varepsilon h}-x^*)/\varepsilon$ converges in
$C([0,T];L^1(\Om))$ to the unique weak solution $z$ of
\eqref{eq:linearized_state}--\eqref{eq:linearized_boundary}.
\end{proposition}

\subsection{Nonlocal adjoint and its $BV$ interpretation}
\label{subsec:formal_adjoint}

We use a current-value costate $\lambda$, related to the present-value costate by
$\lambda_{\mathrm{pv}}=e^{-rt}\lambda$, which produces the $-r\lambda$ term below.
Dualizing the $\delta E$ terms produces the formal scalar feedback
\begin{equation}\label{eq:Gamma_formal}
\Gamma(t):=
\int_{l_0}^{l_m}
\big[\partial_Eg(E^*,\xi)\partial_\xi\lambda(t,\xi)
-\partial_E\mu(E^*,\xi)\lambda(t,\xi)\big]x^*(t,\xi)\dd\xi,
\end{equation}
giving the formal current-value adjoint
\[
-\partial_t\lambda-g(E^*,l)\partial_l\lambda
=
c\,u^*-\big(r+\mu(E^*,l)+u^*\big)\lambda+\chi(l)\Gamma(t),
\qquad
\lambda(T,\cdot)=0,\ \lambda(t,l_m)=0.
\]
Since $x^*(t,\cdot)\in BV(\Om)$ and $\lambda$ is only continuous, the derivative
$\partial_\xi\lambda$ in \eqref{eq:Gamma_formal} is interpreted via a $BV$
integration-by-parts. With $G_t(l):=\partial_Eg(E^*(t),l)\in C^1(\Om)$, the product
$G_tx^*(t,\cdot)\in BV(\Om)$ with $D_l(G_tx^*)=G_tD_lx^*+\partial_lG_t\,x^*\dd l$.

\begin{definition}[Nonlocal adjoint functional]
\label{def:Gamma_BV}
For $\lambda(t,\cdot)\in C(\Om)$,
\begin{equation}\label{eq:Gamma_BV}
\begin{aligned}
\mathcal G_{x^*}[\lambda](t)
:={}&
-\partial_Eg(E^*,l_0)\lambda(t,l_0)x^*(t,l_0^+)
-\int_{[l_0,l_m]}\lambda(t,\xi)\partial_Eg(E^*,\xi)\,D_lx^*(t,\dd\xi)\\
&-\int_{l_0}^{l_m}\lambda(t,\xi)
\big[\partial_\xi\partial_Eg(E^*,\xi)+\partial_E\mu(E^*,\xi)\big]x^*(t,\xi)\dd\xi.
\end{aligned}
\end{equation}
\end{definition}

When $x^*,\lambda$ are smooth, $\mathcal G_{x^*}[\lambda]$ agrees with
\eqref{eq:Gamma_formal}, the $l_m$-boundary term vanishing by $\lambda(t,l_m)=0$.

\begin{lemma}[Boundedness]
\label{lem:Gamma_bounded}
There is $C_\Gamma>0$ with
$|\mathcal G_{x^*}[\lambda](t)|\le C_\Gamma\|\lambda(t,\cdot)\|_{C(\Om)}$ for a.e.\
$t$.
\end{lemma}

The adjoint used below is
\begin{equation}\label{eq:adjoint_BV}
-\partial_t\lambda-g(E^*,l)\partial_l\lambda
=
c\,u^*-\big(r+\mu(E^*,l)+u^*\big)\lambda+\chi(l)\mathcal G_{x^*}[\lambda](t),
\qquad
\lambda(T,\cdot)=0,\ \lambda(t,l_m)=0.
\end{equation}

\begin{assumption}[Continuous transient adjoint]
\label{ass:adjoint_exists}
The nonlocal adjoint equation \eqref{eq:adjoint_BV} admits a unique solution
$\lambda^*\in C([0,T];C(\Om))$, absolutely continuous along adjoint
characteristics.
\end{assumption}

Assumption~\ref{ass:adjoint_exists} is what \eqref{eq:costate_regularity} encodes;
we state it as a hypothesis rather than a theorem because, as
Remark~\ref{rmk:costate_discontinuity} shows, it can fail under the spatial-$BV$
control class and we do not prove a coefficient/control condition that forces it.
Under Assumption~\ref{ass:adjoint_exists} the functional
$\mathcal G_{x^*}[\lambda^*]$ pairs the continuous $\lambda^*$ with the finite
measure $D_lx^*$ and is well defined (Lemma~\ref{lem:Gamma_bounded}), and the
gradient representation and finite-horizon single-crossing results below are
valid. We record what a genuine proof would require.

\begin{remark}[Toward a non-assumed continuous adjoint]
\label{rmk:adjoint_sufficient}
A genuine well-posedness theorem would assume a property of the coefficients and of
$u^*$ ensuring that, for every prescribed scalar source $\gamma\in C([0,T])$, the
local terminal-value transport solution $\lambda[\gamma]$ lies in
$C([0,T];C(\Om))$ and that $\gamma\mapsto\lambda[\gamma]$ is bounded
$C([0,T])\to C([0,T];C(\Om))$; the contraction
$\gamma\mapsto\mathcal G_{x^*}[\lambda[\gamma]]$ (which has a short-time factor
from the characteristic interval, by Lemma~\ref{lem:Gamma_bounded}) would then
deliver a continuous $\lambda^*$. The obstruction is exactly the carriage of a
jump of $c\,u^*$ along an adjoint characteristic to a fixed terminal size
(Remark~\ref{rmk:costate_discontinuity}); a noncharacteristic-transversality
condition on the switching curves of $u^*$ removes it, but verifying such a
condition is a regularity question for the optimal control that we do not settle
here. Recent structured-population work likewise emphasizes that differentiability
and regularity of these solution operators are delicate
\cite{diekmann2025age}.
\end{remark}

\begin{remark}[The costate need not be continuous in size; necessity of \eqref{eq:costate_regularity}]
\label{rmk:costate_discontinuity}
A uniform spatial-$BV$ bound on $u^*$ does not by itself force
$\lambda^*(t,\cdot)\in C(\Om)$. Take constant speed $g\equiv1$ on $\Om$ and the
one-switch control $u^*(s,l)=\mathbf 1_{\{l-s>\ell_c\}}$, which has spatial
variation $1$ for every $s$, hence lies in $\UadBV$. Along the adjoint
characteristic $l(s)=l+s-t$ one has $u^*(s,l(s))=\mathbf 1_{\{l-t>\ell_c\}}$,
constant in $s$ but discontinuous as a function of the terminal size $l$ at
$l=t+\ell_c$. The characteristic integral defining $\lambda^*(t,l)$ therefore
inherits a jump at $l=t+\ell_c$, so $\lambda^*(t,\cdot)\notin C(\Om)$ in general.
Consequently \eqref{eq:costate_regularity} is a genuine restriction, not a
consequence of the admissible class. It can be guaranteed by, e.g.,
restricting to controls continuous in size, or by a noncharacteristic-transversality
condition on the switching curves of $u^*$ (so that no jump of $u^*$ is carried
undamped along an adjoint characteristic to a fixed terminal size). When
\eqref{eq:costate_regularity} fails, the finite-horizon costate must be handled in
$BV(\Om)$ with a precise representative and a measure--$BV$ pairing in place of
\eqref{eq:Gamma_BV}; we do not develop that here, and the affected results
(Definition~\ref{def:Gamma_BV}, Assumption~\ref{ass:adjoint_exists},
continuity of $S^*$, and the finite-horizon single-crossing formulation) are
stated under \eqref{eq:costate_regularity}. None of the stationary results of
Section~\ref{sec:stationary_switching} depends on it.
\end{remark}

\subsection{First variation}
\label{subsec:adjoint_identity}

\begin{proposition}[Adjoint representation]
\label{prop:gradient_formula}
For every bounded feasible direction $h$ of finite spatial variation,
\[
J_T'(u^*)h
=
\int_0^T e^{-rt}\int_\Om x^*(t,l)S^*(t,l)h(t,l)\dd l\dd t,
\qquad
S^*(t,l):=c(l)-\lambda^*(t,l).
\]
\end{proposition}

\subsection{Variational inequality and bang--bang}
\label{subsec:variational_inequality}

\begin{theorem}[First-order variational inequality]
\label{thm:variational_inequality}
For optimal $u^*\in\UadBV$, with $S^*=c-\lambda^*$,
\begin{equation}\label{eq:variational_inequality}
\int_0^T e^{-rt}\int_\Om x^*S^*\big(v-u^*\big)\dd l\dd t\le0
\qquad\text{for all }v\in\UadBV.
\end{equation}
\end{theorem}

If the variation budget is active, \eqref{eq:variational_inequality} cannot in
general be reduced to a pointwise rule, and one writes the optimality condition
abstractly as $e^{-rt}x^*S^*\in N_{\UadBV}(u^*)$, the normal cone for the
maximization convention.

\begin{theorem}[Bang--bang under an inactive constraint]
\label{thm:bang_bang}
Suppose the variation constraint is inactive with a uniform margin,
\begin{equation}\label{eq:inactive_BV}
\operatorname*{ess\,sup}_{t}\TV_\Om(u^*(t,\cdot))\le C_u-\eta,\qquad\eta>0.
\end{equation}
Then for a.e.\ $(t,l)$,
$u^*(t,l)\in\argmax_{v\in[0,u_{\max}]}v\,x^*(t,l)S^*(t,l)$, and where
$x^*(t,l)>0$,
\begin{equation}\label{eq:bang_bang_rule}
u^*(t,l)=
\begin{cases}
0, & S^*(t,l)<0,\\
u_{\max}, & S^*(t,l)>0.
\end{cases}
\end{equation}
On the singular set $\Sigma:=\{S^*=0\}$ the first-order condition does not
determine $u^*$.
\end{theorem}

\begin{remark}[Active constraint; singular arcs]
\label{rmk:active_variation}
If the budget is active, the normal cone contains a contribution from the $BV$
constraint and \eqref{eq:bang_bang_rule} need not follow; the general condition
remains \eqref{eq:variational_inequality}. On $\Sigma$, determining the control
requires higher-order or structural arguments, which we do not pursue.
\end{remark}

\subsection{Bang--bang does not imply a threshold}
\label{subsec:bang_bang_not_threshold}

With $\mathcal H_t:=\{l:S^*(t,l)>0\}$, the bang--bang rule gives
$u^*(t,\cdot)=u_{\max}\mathbf 1_{\mathcal H_t}$ off $\Sigma$, but $\mathcal H_t$
may be several disjoint intervals. A minimum-size policy requires the stronger
$\mathcal H_t=(L^*(t),l_m]$. Thus
\[
\text{inactive $BV$ constraint}+\text{maximum principle}
\ \Longrightarrow\
\text{bang--bang harvesting,}
\]
and the threshold conclusion needs the single-crossing analysis of
Section~\ref{sec:threshold}. Under the size-continuity hypothesis
\eqref{eq:costate_regularity}, $S^*=c-\lambda^*$ is continuous in $l$, and the
single-crossing geometry is then the only additional hypothesis; absent
\eqref{eq:costate_regularity}, $S^*(t,\cdot)$ may jump (Remark~\ref{rmk:costate_discontinuity}),
and the threshold formulation must be read for a precise representative of
$S^*$.

% ============================================================================
\section{Bang--bang versus threshold harvesting: the single-crossing condition}
\label{sec:threshold}

Under an inactive constraint, Theorem~\ref{thm:bang_bang} gives the bang--bang law
\eqref{eq:bang_bang_rule}; we now ask when it is a minimum-size policy. The
distinction is
\[
\begin{aligned}
\text{bang--bang}
&\ \Longleftrightarrow\ u^*(t,l)\in\{0,u_{\max}\}\ \text{a.e.},\\
\text{minimum-size}
&\ \Longleftrightarrow\ u^*(t,l)=u_{\max}\mathbf 1_{\{l>L^*(t)\}}\ \text{a.e.}
\end{aligned}
\]

\subsection{Upward single crossing}
\label{subsec:single_crossing_definition}

\begin{definition}[Weak upward single crossing]
\label{def:weak_single_crossing}
$S:\Om\to\R$ has the weak upward single-crossing property if the ordered sign
condition
\begin{equation}\label{eq:ordered_sign}
l_1<l_2\ \Longrightarrow\ \operatorname{sgn}S(l_1)\le\operatorname{sgn}S(l_2),
\qquad\operatorname{sgn}S\in\{-1,0,1\},
\end{equation}
holds; equivalently, $\{l:S(l)<0\}$ is a lower set and $\{l:S(l)>0\}$ is an
upper set.
\end{definition}

\begin{remark}[Why the one-sided implication is insufficient]
\label{rmk:single_crossing_pitfall}
The condition $S(l_1)>0\Rightarrow S(l_2)\ge0$ for $l_1<l_2$ does not
characterize single crossing: a continuous $S$ that is positive on $(l_0,a)$ and
then drops to and stays at $0$ satisfies it, yet $\{S>0\}=(l_0,a)$ is not an upper
set, and the sign pattern is positive-then-zero rather than
negative$\,|\,$zero$\,|\,$positive. Requiring both that $\{S<0\}$ be a
lower set and that $\{S>0\}$ be an upper set rules this out and is what
Lemma~\ref{lem:sign_decomposition} uses.
\end{remark}

\begin{definition}[Strict upward single crossing]
\label{def:strict_single_crossing}
A continuous $S$ has the strict property if either $S<0$ on $\Om$, or $S>0$ on
$\Om$, or there is a unique $L\in(l_0,l_m)$ with $S<0$ on $(l_0,L)$, $S(L)=0$,
$S>0$ on $(L,l_m)$.
\end{definition}

\begin{assumption}[Single crossing of $S^*$]
\label{ass:single_crossing}
For a.e.\ $t$, $l\mapsto S^*(t,l)$ is continuous and weakly upward single crossing.
\end{assumption}

Assumption~\ref{ass:single_crossing} does not follow from the maximum principle;
its validity depends on $c$, the state, the rates, and the full nonlocal adjoint.
In the finite-horizon problem we retain it as an explicit hypothesis; the results
that avoid it are the stationary ones of
Section~\ref{sec:stationary_switching}, unconditional only within the stationary
canonical system.

\subsection{Geometry and threshold structure}
\label{subsec:single_crossing_geometry}

For fixed $t$, let $L_-(t):=\sup\{l:S^*(t,l)<0\}$, $L_+(t):=\inf\{l:S^*(t,l)>0\}$
(with $\sup\varnothing=l_0$, $\inf\varnothing=l_m$).

\begin{lemma}[Sign decomposition]
\label{lem:sign_decomposition}
Let $S^*(t,\cdot)$ be continuous and weakly upward single crossing in the sense of
Definition~\ref{def:weak_single_crossing}. Then $L_-(t)\le L_+(t)$ and
\begin{equation}\label{eq:sign_decomp}
S^*(t,l)<0\ (l<L_-),\qquad
S^*(t,l)=0\ (L_-<l<L_+),\qquad
S^*(t,l)>0\ (l>L_+),
\end{equation}
with the endpoints determined as follows: if both sign sets $\{S^*<0\}$ and
$\{S^*>0\}$ are nonempty, continuity gives $S^*(L_-)=S^*(L_+)=0$; if one sign set
is empty, the corresponding endpoint coincides with a domain endpoint and $S^*$
may take a strict sign there (e.g.\ $L_-=L_+=l_0$ and $S^*(l_0)>0$ when $S^*>0$
throughout, or $L_-=L_+=l_m$ and $S^*(l_m)<0$ when $S^*<0$ throughout). Conversely,
for continuous $S^*$ the existence of such $L_-\le L_+$ with \eqref{eq:sign_decomp}
is equivalent to the ordered sign condition \eqref{eq:ordered_sign}.
\end{lemma}

\begin{theorem}[Weak threshold structure]
\label{thm:weak_threshold_structure}
Under Theorem~\ref{thm:bang_bang} and Assumption~\ref{ass:single_crossing}, for
a.e.\ $t$ exactly one of the following holds, where $x^*(t,\cdot)>0$:
\begin{enumerate}[label=\textnormal{(\roman*)}]
\item $S^*(t,\cdot)\le0$ ($\{S^*>0\}=\varnothing$): $u^*(t,\cdot)=0$
off the zero set $\{S^*(t,\cdot)=0\}$;
\item $S^*(t,\cdot)\ge0$ ($\{S^*<0\}=\varnothing$):
$u^*(t,\cdot)=u_{\max}$ off the zero set;
\item both signs occur: with $l_0<L_-(t)\le L_+(t)<l_m$,
\[
u^*(t,l)=
\begin{cases}
0, & l<L_-(t),\\
u_{\max}, & l>L_+(t),
\end{cases}
\]
and $u^*$ undetermined on the singular interval
$(L_-(t),L_+(t))\subseteq\{S^*(t,\cdot)=0\}$, on whose endpoints $S^*=0$.
\end{enumerate}
\end{theorem}

\begin{theorem}[Strict minimum-size policy]
\label{thm:strict_threshold_structure}
Under Theorem~\ref{thm:bang_bang}, if $S^*(t,\cdot)$ is strictly upward single
crossing for a.e.\ $t$, then exactly one holds: (i) $S^*<0$ and $u^*=0$;
(ii) $S^*>0$ and $u^*=u_{\max}$; (iii) there is a unique $L^*(t)\in(l_0,l_m)$ with
$S^*(t,L^*(t))=0$ and
$\displaystyle u^*(t,l)=u_{\max}\mathbf 1_{\{l>L^*(t)\}}$ a.e.
\end{theorem}

\[
\text{bang--bang optimality}+\text{strict upward single crossing}
\ \Longrightarrow\ \text{minimum-size harvesting.}
\]

\subsection{Measurability and regularity}
\label{subsec:threshold_regularity}

\begin{proposition}[Measurability]
\label{prop:threshold_measurable}
If $S^*$ is measurable in $t$, continuous in $l$, and weakly upward single
crossing in $l$ for a.e.\ $t$, then $L_+$ and $L_-$ are measurable.
\end{proposition}

\begin{proposition}[Local regularity under transversal crossing]
\label{prop:threshold_regular}
If $S^*\in C^1$ near $(t_0,L^*(t_0))$ with $S^*=0$ and
$\partial_lS^*\neq0$ there, then locally there is a unique $C^1$ trajectory $L^*$
with $S^*(t,L^*(t))=0$ and
\begin{equation}\label{eq:threshold_velocity}
\dot L^*(t)=-\frac{\partial_tS^*(t,L^*(t))}{\partial_lS^*(t,L^*(t))}.
\end{equation}
\end{proposition}

Equation \eqref{eq:threshold_velocity} is a local identity, not an independent
evolution law, and fails at tangential or multiple crossings.

\subsection{Monotonicity is sufficient but not necessary}
\label{subsec:monotonicity_not_necessary}

\begin{proposition}[Single crossing is strictly weaker than monotonicity]
\label{prop:single_crossing_weaker}
If $S$ is strictly increasing it is strictly upward single crossing whenever it
changes sign. The converse fails.
\end{proposition}

\subsection{A robust single-crossing certificate}
\label{subsec:robust_single_crossing}

\begin{proposition}[Quantitative certificate]
\label{prop:quantitative_single_crossing}
Let $S\in C^1(\Om)$ and suppose there are $L\in(l_0,l_m)$, $\delta,m,\kappa>0$ with
\[
S(l)\le-m,\ l\in[l_0,L-\delta];\quad
S(l)\ge m,\ l\in[L+\delta,l_m];\quad
S'(l)\ge\kappa,\ l\in[L-\delta,L+\delta].
\]
Then $S$ has exactly one zero in $(L-\delta,L+\delta)$, an upward crossing.
Moreover any $\widetilde S\in C^1$ with
$\|\widetilde S-S\|_{L^\infty}<m$ and
$\|\widetilde S'-S'\|_{L^\infty(L-\delta,L+\delta)}<\kappa$ also has exactly one
upward crossing.
\end{proposition}

\subsection{Switching geometries and the stationary route}
\label{subsec:switching_geometries}

For fixed $t$, $S^*$ may give: (i) no harvest ($S^*<0$); (ii) full harvest
($S^*>0$); (iii) a minimum-size regime; (iv) a singular band ($S^*\equiv0$ on an
interval); or (v) a multiple-switch regime ($\mathcal H_t$ disconnected). The last
is the central caution: nonlocal feedback can alter the marginal value of one size
class through its effect on all others, reshaping $S^*$ enough to create extra
crossings while the control stays bang--bang. Establishing single crossing
uniformly in time would require control of the sign geometry of the nonlocal
costate, so we retain it as a hypothesis in the time-dependent theory. In the
stationary problem the adjoint is a linear ODE with scalar nonlocal feedback,
which permits the exact decomposition $S=S_{\mathrm{red}}-\Gamma\psi$ developed
next.

% ============================================================================
\section{Stationary nonlocal switching correction}
\label{sec:stationary_switching}

\subsection{Which problem the stationary adjoint belongs to}
\label{subsec:stationary_problem_statement}

The time-independent adjoint studied in this section carries a discount term
$r\lambda$ and therefore is not the adjoint of the finite-horizon problem
$J_T$ of Section~\ref{sec:control_existence}, whose costate satisfies the terminal
condition $\lambda(T,\cdot)=0$ and is generically nonstationary; nor is it the
Lagrange multiplier of the static equilibrium-yield problem
$\max_u\int_\Om c\,u\,x^u$, in which $r$ does not appear (multiplying a permanently
stationary trajectory's value $\int_0^\infty e^{-rt}\mathcal Y(u)\dd t=\mathcal
Y(u)/r$ by $1/r$ does not introduce an $r\lambda$ term). The natural reading is the
infinite-horizon discounted problem
\begin{equation}\label{eq:Jinfty}
J_\infty(u)=\int_0^\infty e^{-rt}\int_\Om c(l)\,u(t,l)\,x(t,l)\dd l\dd t,
\qquad r>0,
\end{equation}
whose stationary canonical extremals are time-independent triples
$(x^*,E^*,u)$ together with a current-value costate $\lambda$ solving the state
equation, the current-value adjoint equation, and the pointwise maximum condition.
Equation \eqref{eq:full_stationary_adjoint} below is exactly the stationary
current-value canonical adjoint of \eqref{eq:Jinfty} (equivalently, a
resolvent/spectral-shift equation with parameter $r\ge0$). This section analyzes
the switching geometry of such stationary canonical extremals; it does not
by itself prove their global optimality, nor that a stationary extremal is optimal
among nonstationary controls. Accordingly we speak throughout of stationary
canonical bang--bang extremals, and a ``minimum-size policy'' conclusion is a
statement about such an extremal, not a proved solution of a dynamic optimization.

\begin{definition}[Stationary canonical candidate]
\label{def:canonical_candidate}
A quadruple $(x^*,E^*,u,\lambda)$ is a stationary canonical candidate if
$E^*=\Phi(E^*)$ with profile $x^*=x_{E^*}$, $\lambda$ solves the stationary
current-value adjoint \eqref{eq:full_stationary_adjoint}, and the maximum
(consistency) condition
\begin{equation}\label{eq:stationary_consistency}
u(l)\in\argmax_{v\in[0,u_{\max}]}v\,x^*(l)\,S(l),
\qquad S:=c-\lambda,
\end{equation}
holds for a.e.\ $l$. The rank-one reduction below is valid for an
arbitrary prescribed stationary $u$; only the control-policy conclusions
(that the harvested set is an upper interval) additionally invoke
\eqref{eq:stationary_consistency}.
\end{definition}

In the stationary problem the nonlocal feedback is scalar, allowing an exact
rank-one reduction of the adjoint. We develop the reduction, small-feedback and
certificate results, and the unification identity $B|_{r=0}=\Phi'(E^*)$.

\subsection{Operating point}
\label{subsec:stationary_operating_point}

Fix a stationary policy $u=u(l)\in[0,u_{\max}]$ and a stationary environment $E^*$
with $E^*=\Phi(E^*)$ and profile $x^*=x_{E^*}$ as in
\eqref{eq:stationary_profile}. Unlike the transient problem, the stationary
profile and stationary costate are absolutely continuous in size: since
$(\bar g x^*)'=-(\bar\mu+u)x^*\in L^\infty(\Om)$ and $\bar g\ge g_{\min}$ one has
$x^*\in W^{1,\infty}(\Om)$, and the stationary terminal-value ODEs below give
$\lambda,\psi\in W^{1,\infty}(\Om)$. We therefore define the stationary feedback
functional directly by its strong form, with no measure pairing and no
appeal to the transient continuity hypothesis \eqref{eq:costate_regularity}. Write
\[
\bar g:=g(E^*,\cdot),\quad
\bar\mu:=\mu(E^*,\cdot),\quad
g_E:=\partial_Eg(E^*,\cdot),\quad
\mu_E:=\partial_E\mu(E^*,\cdot),
\]
and the discounted-mortality coefficient (this is the symbol freed by the
renaming in Section~\ref{sec:stationary}), defined for $r\ge0$ by
\[
a_r(l):=r+\bar\mu(l)+u(l)\ge0,
\qquad\text{with }a_r\ge r>0\text{ for }r>0.
\]
The original discounted control problem has $r>0$; the value $r=0$ is used below
only as the zero-shift limit in the closure-gain identity, and the terminal-value
ODEs remain well posed there because $\bar g\ge g_{\min}>0$. We write $a_r$
throughout to make the dependence of $\psi_r$, $B(r)$, and the spectral shift
explicit, abbreviating $a:=a_r$ when $r$ is fixed.
For $\zeta\in W^{1,\infty}(\Om)$ define the stationary feedback functional by
\begin{equation}\label{eq:stationary_G_smooth}
\mathcal G^*[\zeta]
:=
\int_{l_0}^{l_m}\big(g_E(\xi)\zeta'(\xi)-\mu_E(\xi)\zeta(\xi)\big)x^*(\xi)\dd\xi,
\end{equation}
which is well defined and bounded on $W^{1,\infty}(\Om)$ because
$x^*\in W^{1,\infty}(\Om)\subset L^\infty(\Om)$ and $g_E,\mu_E\in L^\infty(\Om)$.
When $x^*$ is only $BV$ (the transient situation), an integration by parts gives
the equivalent $BV$ representation
\begin{equation}\label{eq:stationary_G_functional}
\mathcal G^*[\zeta]
=
-g_E(l_0)\zeta(l_0)x^*(l_0^+)
-\int_{[l_0,l_m]}\zeta(\xi)g_E(\xi)\,D_lx^*(\dd\xi)
-\int_{l_0}^{l_m}\zeta(\xi)\big(g_E'(\xi)+\mu_E(\xi)\big)x^*(\xi)\dd\xi,
\end{equation}
agreeing with \eqref{eq:stationary_G_smooth} for $x^*\in W^{1,\infty}$; we use the
strong form \eqref{eq:stationary_G_smooth} as the foundation of this section and
regard \eqref{eq:stationary_G_functional} as an equivalent representation. The
full stationary current-value adjoint is
\begin{equation}\label{eq:full_stationary_adjoint}
\bar g(l)\lambda'(l)=a(l)\lambda(l)-c(l)u(l)-\chi(l)\Gamma,
\qquad
\lambda(l_m)=0,
\qquad
\Gamma=\mathcal G^*[\lambda].
\end{equation}

\subsection{Reduced adjoint and response function}
\label{subsec:reduced_auxiliary}

Define $\lambda_{\mathrm{red}}$ and $\psi$ by
\begin{align}
\bar g\lambda_{\mathrm{red}}'&=a\lambda_{\mathrm{red}}-cu,
&\lambda_{\mathrm{red}}(l_m)&=0,
\label{eq:reduced_adjoint}\\
\bar g\psi'&=a\psi-\chi,
&\psi(l_m)&=0,
\label{eq:auxiliary_psi}
\end{align}
with explicit solutions
\begin{equation}\label{eq:lambda_red_formula}
\lambda_{\mathrm{red}}(l)=\int_l^{l_m}\frac{c(s)u(s)}{\bar g(s)}
\exp\!\Big(-\!\int_l^s\frac{a}{\bar g}\Big)\dd s,
\qquad
\psi(l)=\int_l^{l_m}\frac{\chi(s)}{\bar g(s)}
\exp\!\Big(-\!\int_l^s\frac{a}{\bar g}\Big)\dd s,
\end{equation}
both nonnegative for $c,u,\chi\ge0$. Set
$S_{\mathrm{red}}(l):=c(l)-\lambda_{\mathrm{red}}(l)$.

\begin{remark}[Interpretation of $\psi$]
\label{rmk:psi_interpretation}
$\psi$ is the costate response to one unit of the scalar nonlocal source; it
propagates the environmental correction backward from larger sizes. The full
nonlocal effect has fixed shape $\psi$, with magnitude and sign set by the scalar
$\Gamma$.
\end{remark}

\subsection{Exact rank-one correction}
\label{subsec:rank_one_correction}

Define
\begin{equation}\label{eq:A_B_definition}
A:=\mathcal G^*[\lambda_{\mathrm{red}}],
\qquad
B:=\mathcal G^*[\psi],
\end{equation}
with smooth forms
\[
A=\int_{l_0}^{l_m}\!\big(g_E\lambda_{\mathrm{red}}'-\mu_E\lambda_{\mathrm{red}}\big)x^*\dd\xi,
\qquad
B=\int_{l_0}^{l_m}\!\big(g_E\psi'-\mu_E\psi\big)x^*\dd\xi.
\]

\begin{theorem}[Rank-one correction]
\label{thm:rank_one_correction}
If $1-B\neq0$, the full stationary adjoint
\eqref{eq:full_stationary_adjoint} has the unique representation
\[
\lambda=\lambda_{\mathrm{red}}+\Gamma\psi,
\qquad
\Gamma=\frac{A}{1-B},
\]
so that
\begin{equation}\label{eq:rank_one_switching}
S=S_{\mathrm{red}}-\frac{A}{1-B}\psi,
\qquad
S'=S_{\mathrm{red}}'-\frac{A}{1-B}\psi'.
\end{equation}
\end{theorem}

\begin{remark}[Sherman--Morrison and nonresonance]
\label{rmk:nonresonance}
The closure $\Gamma=A+\Gamma B$ of Theorem~\ref{thm:rank_one_correction} is the
scalar Sherman--Morrison identity \cite{sherman1950adjustment} made precise as follows. Let $L_0$ denote the
reduced terminal-value operator, so that $\lambda_{\mathrm{red}}=L_0^{-1}(cu)$ and
$\psi=L_0^{-1}\chi$, and let $\mathcal G^*$ act as a bounded linear functional.
The full adjoint solves the rank-one-perturbed equation
\[
\big(I-\psi\otimes\mathcal G^*\big)\lambda=\lambda_{\mathrm{red}},
\qquad
(\psi\otimes\mathcal G^*)\lambda:=\psi\,\mathcal G^*[\lambda],
\]
whose inverse is $\big(I-\psi\otimes\mathcal G^*\big)^{-1}
=I+\frac{\psi\otimes\mathcal G^*}{1-\mathcal G^*[\psi]}$, with denominator
$1-\mathcal G^*[\psi]=1-B$. Applying it to $\lambda_{\mathrm{red}}$ reproduces
$\lambda=\lambda_{\mathrm{red}}+\frac{A}{1-B}\psi$ since
$\mathcal G^*[\lambda_{\mathrm{red}}]=A$. If $B=1$ and $A\neq0$ the closure is
inconsistent; if $B=1$ and $A=0$ the adjoint is algebraically nonunique. Hence
uniqueness forces $B\neq1$. The correction is rank one for every $B\neq1$; rank
one does not mean small, since $|\Gamma|=|A|/|1-B|$ can be large near resonance.
\end{remark}

\subsection{Quantitative feedback bounds}
\label{subsec:feedback_bounds}

From \eqref{eq:rank_one_switching},
\[
\|S-S_{\mathrm{red}}\|_{L^\infty}\le\frac{|A|}{|1-B|}\|\psi\|_{L^\infty},
\qquad
\|S'-S_{\mathrm{red}}'\|_{L^\infty}\le\frac{|A|}{|1-B|}\|\psi'\|_{L^\infty},
\]
with $\psi'=(a\psi-\chi)/\bar g$, hence
$\|\psi'\|_\infty\le(\|a\|_\infty\|\psi\|_\infty+\|\chi\|_\infty)/g_{\min}$ and
$0\le\psi\le\|\chi\|_\infty(l_m-l_0)/g_{\min}$.

\begin{lemma}[Feedback vanishes with the sensitivities]
\label{lem:feedback_vanishes}
If a uniformly bounded family of operating points satisfies
$\|g_E\|_{C^1}+\|\mu_E\|_{L^\infty}\le\varepsilon$, then there is $C>0$ with
$|A|+|B|\le C\varepsilon$, and for small $\varepsilon$,
$|\Gamma|\le C\varepsilon/(1-C\varepsilon)=O(\varepsilon)$.
\end{lemma}

\subsection{Small-feedback preservation of single crossing}
\label{subsec:small_feedback}

\begin{theorem}[Monotone case]
\label{thm:small_feedback_monotone}
If $S_{\mathrm{red}}'\ge\eta>0$ on $(l_0,l_m)$, $|B|\le\beta<1$, and
$\frac{|A|}{1-\beta}\|\psi'\|_{L^\infty}<\eta$, then $S'>0$, so $S$ has at most one
zero; if also $S(l_0)<0<S(l_m)$, then $S$ has exactly one upward crossing. If, in
addition, the operating point is a stationary canonical candidate
(Definition~\ref{def:canonical_candidate}, i.e.\ the consistency condition
\eqref{eq:stationary_consistency} holds), the corresponding stationary canonical
bang--bang extremal is a minimum-size policy.
\end{theorem}

\begin{theorem}[Nonmonotone case]
\label{thm:small_feedback_single_crossing}
Suppose there are $L_{\mathrm{red}}\in(l_0,l_m)$, $\delta,m,\kappa>0$ with
$S_{\mathrm{red}}\le-m$ on $[l_0,L_{\mathrm{red}}-\delta]$,
$S_{\mathrm{red}}\ge m$ on $[L_{\mathrm{red}}+\delta,l_m]$, and
$S_{\mathrm{red}}'\ge\kappa$ on $[L_{\mathrm{red}}-\delta,L_{\mathrm{red}}+\delta]$.
If $B\neq1$,
$\frac{|A|}{|1-B|}\|\psi\|_{L^\infty}<m$, and
$\frac{|A|}{|1-B|}\|\psi'\|_{L^\infty([L_{\mathrm{red}}-\delta,L_{\mathrm{red}}+\delta])}<\kappa$,
then $S$ has exactly one upward crossing in
$(L_{\mathrm{red}}-\delta,L_{\mathrm{red}}+\delta)$; if the operating point is a
stationary canonical candidate (consistency condition
\eqref{eq:stationary_consistency}), the corresponding stationary canonical
bang--bang extremal is a minimum-size policy.
\end{theorem}

\begin{corollary}[Weak feedback]
\label{cor:weak_feedback}
Under Lemma~\ref{lem:feedback_vanishes}, if $S_{\mathrm{red}}$ has a robust upward
crossing (the margins of Theorem~\ref{thm:small_feedback_single_crossing}), there
is $\varepsilon_0>0$ such that every operating point with
$\|g_E\|_{C^1}+\|\mu_E\|_\infty<\varepsilon_0$ has a full switching function with
exactly one upward crossing.
\end{corollary}

\subsection{Threshold displacement}
\label{subsec:threshold_displacement}

\begin{proposition}[First-order displacement]
\label{prop:threshold_displacement}
If $S_{\mathrm{red}}(L_{\mathrm{red}})=0$,
$S_{\mathrm{red}}'(L_{\mathrm{red}})\neq0$, and the correction is small enough for
a unique nearby zero $L^*$ to exist, then
$\displaystyle L^*-L_{\mathrm{red}}
=
\frac{\Gamma\psi(L_{\mathrm{red}})}{S_{\mathrm{red}}'(L_{\mathrm{red}})}+O(\Gamma^2)$.
\end{proposition}

The first-order displacement of the switching threshold is illustrated in Figure~\ref{fig:displacement}. The spy inset highlights the horizontal shift $L^* - L_{\mathrm{red}}$ caused by the vertical rank-one perturbation $-\Gamma\psi(l)$, consistently with the approximation in Proposition~8.7.

\begin{figure}[htbp]
    \centering
    \includegraphics[width=0.7\textwidth]{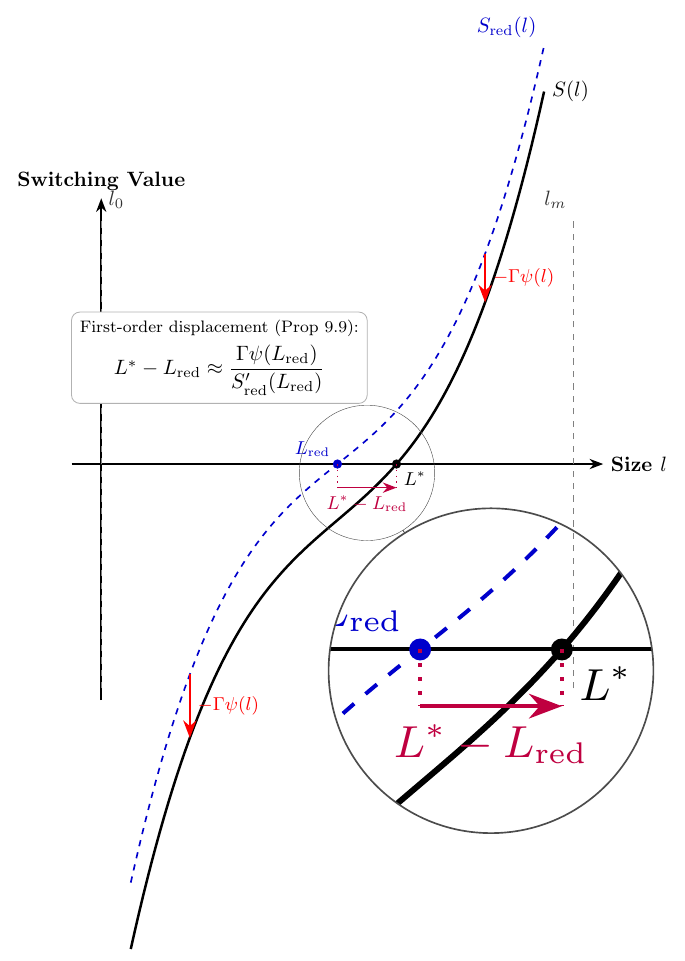}
    \caption{Displacement of switching thresholds. The spy inset magnifies the first-order shift $L^* - L_{\mathrm{red}}$ between $S(l)$ and $S_{\mathrm{red}}(l)$ induced by the rank-one perturbation $-\Gamma\psi(l)$.}
    \label{fig:displacement}
\end{figure}

\subsection{Unification: the closure gain is the zero-discount feedback gain}
\label{subsec:unification}

We now prove that the feedback gain $B$ and the closure derivative $\Phi'(E^*)$ of
Section~\ref{sec:stationary} are the same scalar at zero discount, with $r$ acting
as a spectral shift. Write $\alpha=-g_E/\bar g\ge0$ and $b$ as in
\eqref{eq:a_b_definitions}--\eqref{eq:b_expansion}, both evaluated at $E^*$, and let
$\tau(\xi,s):=\int_\xi^s\dd\eta/\bar g(\eta)$ be the growth time from $\xi$ to $s$.
For $r\ge0$ let $\psi_r$ solve \eqref{eq:auxiliary_psi} with $a=r+\bar\mu+u$, and
$B(r):=\mathcal G^*[\psi_r]$ (smooth form \eqref{eq:stationary_G_smooth}). Define
the discounted tail load
\begin{equation}\label{eq:Wr}
W_r(\xi):=\int_\xi^{l_m}\chi(s)x^*(s)\,e^{-r\tau(\xi,s)}\dd s,
\qquad
W_0=W_{E^*}.
\end{equation}

\begin{lemma}[Forward/adjoint survival duality]
\label{lem:duality}
For every $r\ge0$ and $\xi\in\Om$,
$\displaystyle \bar g(\xi)\,\psi_r(\xi)\,x^*(\xi)=W_r(\xi)$.
In particular, at $r=0$, $\bar g\,\psi\,x^*=W_{E^*}$.
\end{lemma}

\begin{theorem}[Closure-gain identity]
\label{thm:closure_gain_identity}
Let
$\mathsf A^*:=\mathsf A(E^*)$ and $\mathsf C^*:=\mathsf C(E^*)$.
For every $r\ge0$,
\begin{equation}\label{eq:Br_general}
B(r)
=
\underbrace{\int_{l_0}^{l_m}\alpha\chi x^*}_{\mathsf A^*}
-\int_{l_0}^{l_m} b(\xi)\,W_r(\xi)\dd\xi
-r\int_{l_0}^{l_m}\frac{\alpha(\xi)}{\bar g(\xi)}\,W_r(\xi)\dd\xi.
\end{equation}

In particular,
\[
B(0)=\mathsf A^*-\mathsf C^*=\Phi'(E^*).
\]
If, in addition,
$\chi x^*$ is continuous and $\bar g$ is continuous and bounded below on $\Om$
(so that $\xi\mapsto\chi(\xi)x^*(\xi)\bar g(\xi)$ is bounded and the inner integrand
below is dominated uniformly), then $\lim_{r\to\infty}B(r)=0$:
\[
B(0)=\Phi'(E^*),
\qquad
\lim_{r\to\infty}B(r)=0.
\]
\end{theorem}

\begin{corollary}[Automatic and unconditional nonresonance]
\label{cor:nonresonance_consequences}
\leavevmode
\begin{enumerate}[label=(\alph*)]
\item If the local condition $\Phi'(E^*)<0$ holds at the operating point, then
$1-B(0)=1-\Phi'(E^*)>1$, and by continuity there is $r_0>0$ with $1-B(r)>0$ for
$r\in[0,r_0]$. This holds in particular whenever the global integrated-monotonicity
condition of Theorem~\ref{thm:integrated_monotonicity} is satisfied, but only the
local sign of $\Phi'$ at $E^*$ is needed here.
\item (Unconditional in $r$.) Since $\alpha,b,W_r\ge0$, the two subtracted terms in
\eqref{eq:Br_general} are nonnegative, so $B(r)\le\mathsf A^*$ for all $r\ge0$.
Hence
\[
\mathsf A^*
=
\int_{l_0}^{l_m}\Big(-\frac{g_E}{\bar g}\Big)\chi x^*\dd l
<1
\quad\Longrightarrow\quad
1-B(r)>0\ \ \text{for all }r\ge0.
\]
\end{enumerate}
\end{corollary}

\begin{remark}[Characteristic-equation reading and its limits]
\label{rmk:characteristic_reading}
Because $B$ carries no payoff ($A$ alone carries the source $cu$), $B(r)$ is the
scalar loop gain of the stationary closure at spectral shift $r$, and
$1-B(r)=0$ has the form of a closure/characteristic equation; at $r=0$ it reduces
to $1-\Phi'(E^*)$. The condition $\Phi'(E^*)<0$ thus controls only the $r=0$
slice, i.e.\ the zero-discount closure. An all-$r$ guarantee is equivalent
to the absence of a nonnegative real root of $1-B(r)$ and does not follow
from $\Phi'(E^*)<0$ in general. Formula \eqref{eq:Br_general} does not imply
monotonicity of $B(r)$ under the present assumptions, because discounting acts
differently on spatially separated contributions: indeed
$B(r)-B(0)=\int b\,(W_0-W_r)-r\int\frac{\alpha}{\bar g}W_r$ has a sign that need
not be definite, the first term being nonnegative and the second nonpositive. The
robust, operating-point-checkable guarantee is therefore
Corollary~\ref{cor:nonresonance_consequences}(b); the neighborhood guarantee of
part (a) is the consequence of steady-state uniqueness alone.
\end{remark}

\subsection{Reconstruction and certificate}
\label{subsec:reconstruction_procedure}

For a prescribed stationary $u$: (1) solve $E^*=\Phi(E^*)$ and form $x^*$; (2)
solve \eqref{eq:reduced_adjoint} for $\lambda_{\mathrm{red}}$; (3) solve
\eqref{eq:auxiliary_psi} for $\psi$; (4) evaluate $A,B$ from
\eqref{eq:A_B_definition}; (5) check $d:=1-B$; (6) if $d\neq0$, set
$\Gamma=A/d$ and $S=S_{\mathrm{red}}-\Gamma\psi$; (7) read the sign geometry of
$S$. The residual
$R_{\mathrm{rank}}:=\|\lambda-\lambda_{\mathrm{red}}-\Gamma\psi\|_{L^\infty}$ is
zero in exact arithmetic and serves as an implementation check.

\begin{definition}[Grid diagnostic]
\label{def:grid_diagnostic}
A grid diagnostic computes $\lambda_{\mathrm{red}},\psi,A,B,\Gamma,S$ on a fine
mesh and records the observed sign changes of $S$. It is evidence for, not a proof
of, single crossing.
\end{definition}

For a rigorous certificate, partition $l_0=\ell_0<\cdots<\ell_N=l_m$,
$I_j=[\ell_{j-1},\ell_j]$, and suppose validated computations produce interval
enclosures $[S](I_j)\supseteq\{S(l):l\in I_j\}$, $[S'](I_j)$, and $[A],[B],[\Gamma]$
with $0\notin1-[B]$.

\begin{theorem}[Validated single-crossing certificate]
\label{thm:validated_certificate}
If there is a unique $k$ with $\sup[S](I_j)<0$ for $j<k$, $\inf[S](I_j)>0$ for
$j>k$, $\inf[S'](I_k)>0$, and $S(\ell_{k-1})<0<S(\ell_k)$, then $S$ has exactly one
zero in $I_k$, an upward crossing; if the operating point is a stationary
canonical candidate \eqref{eq:stationary_consistency}, the corresponding
stationary canonical bang--bang extremal is a minimum-size policy.
\end{theorem}

\begin{corollary}[Endpoint policies]
\label{cor:validated_endpoint_policies}
Suppose the operating point is a stationary canonical candidate
\eqref{eq:stationary_consistency}. If $\sup[S](I_j)<0$ for all $j$, then $S<0$ and
the extremal is no harvesting; if $\inf[S](I_j)>0$ for all $j$, then $S>0$ and the
extremal is full harvesting.
\end{corollary}

\begin{proposition}[Multiple-switch certificate]
\label{prop:validated_multiple_switch}
If for ordered $q_1<q_2<q_3$ validated enclosures give
$S(q_1)<0<S(q_2)$ and $S(q_2)>0>S(q_3)$ (or the reversed pattern), then $S$ has at
least two zeros and fails single crossing.
\end{proposition}

\subsection{Near-resonance and conditioning}
\label{subsec:near_resonance}

The denominator $d=1-B$ must be monitored: for perturbations,
$\displaystyle \delta\Gamma=\frac{\delta A}{1-B}+\frac{A\,\delta B}{(1-B)^2}+\cdots$,
so reconstruction is ill-conditioned as $B\to1$. A certificate should report
$A,B,1-B,\Gamma,|A|/|1-B|$ with validated error bounds. Near resonance, a small
change in the sensitivities produces a large change in $S$: this identifies a
regime of high marginal feedback gain in which the minimum-size policy is
structurally fragile.

\subsection{Summary}
\label{subsec:stationary_correction_summary}

The apparently nonlocal stationary adjoint reduces to two local terminal-value
ODEs and one scalar closure $\Gamma=A+\Gamma B$, giving, under $1-B\neq0$,
\[
\Gamma=\frac{A}{1-B},
\qquad
S=S_{\mathrm{red}}-\frac{A}{1-B}\psi,
\qquad
B|_{r=0}=\Phi'(E^*).
\]
This furnishes an exact reconstruction of the switching function, small-feedback
single-crossing criteria, a threshold-displacement formula, diagnostics for
multiple-switch and no-threshold regimes, a validated certificate, and the
identification of the closure gain with the discounted feedback gain. The reduction
holds for any prescribed stationary control; the minimum-size conclusion applies to
any stationary canonical bang--bang extremal of the infinite-horizon discounted
problem (Definition~\ref{def:canonical_candidate}). The conclusion is not that
nonlocality always preserves a minimum-size extremal, nor that a stationary
extremal is globally optimal, but that the exact scalar correction---anchored to
the classical closure derivative---identifies precisely when single crossing holds
and when it does not.

% ============================================================================
\section{Ecological and numerical consequences}
\label{sec:concrete}

We now instantiate the theory in a calibrated, density-dependent von Bertalanffy \cite{von1938quantitative,de2008simplifying}
model and use it to extract ecological consequences. The aim of this section is
not to confirm the preceding theorems---the rank-one reduction and the identity
$B|_{r=0}=\Phi'(E^*)$ are exact---but to use them as instruments that expose a
concrete, and previously implicit, ecological trade-off and to map the
circumstances under which a conventional minimum-size harvest rule remains valid.

\subsection{Calibrated model}
\label{subsec:calibration}

Lengths are in centimetres and time in years. Growth is von Bertalanffy with
multiplicative crowding suppression, mortality is additive in the environmental
load, and the environment is a biomass-weighted competition index:
\[
g(E,l)=\frac{k(L_\infty-l)}{1+\rho_g E},\quad
\mu(E,l)=\mu_0+\rho_\mu E,\quad
\chi(l)=\Big(\tfrac{l}{l_m}\Big)^{3},\quad
E=\int_{l_0}^{l_m}\chi(l)x(l)\dd l.
\]
Because $L_\infty>l_m$, growth is bounded below by
$g_{\min}(M)=k(L_\infty-l_m)/(1+\rho_g M)>0$ on each bounded environmental interval
$[0,M]$, which is the local form of \eqref{eq:H2}; the stationary analysis uses it
on the closure interval $[0,M_{\mathrm{st}}]$, with $M_{\mathrm{st}}=\Phi(0)$ (the
maximum of the decreasing closure map, computed per parameter set), and the
finite-horizon analysis uses it on $[0,M_T]$. Its infimum over all $E\ge0$ is zero,
so no global floor is claimed; one could impose a global floor by the variant
$g=k(L_\infty-l)[\varepsilon_g+(1-\varepsilon_g)/(1+\rho_gE)]$ with
$\varepsilon_g>0$, which we do not need here. The compensatory signs
\eqref{eq:H4} hold with
$\partial_Eg=-k(L_\infty-l)\rho_g/(1+\rho_gE)^2\le0$ and $\partial_E\mu=\rho_\mu\ge0$.
Fecundity is given by
$m(l)=(l/l_m)^3/(1+e^{-(l-l_{50})/\delta_m})$. The realised value per harvested
individual is taken in two forms: a monotone form $c(l)=(l/l_m)^3$ (value rising
with body weight) and a dome-shaped form
$c(l)=\exp(-((l-l_{\mathrm{val}})/w_{\mathrm{val}})^2)$ representing a harvest
value peaked at an intermediate size $l_{\mathrm{val}}$ (for instance a
size-graded market premium). We emphasise that this dome enters the payoff $c$
only; it is not a gear-selectivity constraint on the feasible fishing mortality.
Baseline parameters and ranges are listed in Table~\ref{tab:params}; all
quantities below are computed on a length grid with the trapezoidal closure and
the exact terminal-value adjoints of Section~\ref{sec:stationary_switching}.

\begin{table}[t]
\centering
\small
\caption{Model parameters: definitions, units, baseline values, plausible ranges,
and basis. The environment $E$ is a dimensionless biomass-weighted competition
index; feedback coefficients $\rho_g,\rho_\mu$ are calibrated so that the
crowding factors $1+\rho_gE$ and the mortality increment $\rho_\mu E$ are of
order one at the operating point.}
\label{tab:params}
\begin{tabular}{lllll}
\toprule
Symbol & Definition & Unit & Baseline & Range \\
\midrule
$l_0,l_m$ & recruitment / exit size & cm & $10,\,100$ & fixed \\
$k$ & von Bertalanffy rate & yr$^{-1}$ & $0.20$ & $0.14$--$0.26$ \\
$L_\infty$ & asymptotic length & cm & $130$ & $121$--$139$$^{\dagger}$ \\
$\mu_0$ & baseline mortality & yr$^{-1}$ & $0.20$ & $0.14$--$0.26$ \\
$\rho_g$ & growth-suppression feedback & -- & $0.30$ & $0$--$3$ \\
$\rho_\mu$ & mortality-increase feedback & -- & $0.15$ & $0$--$0.8$ \\
$u_{\max}$ & maximum fishing mortality & yr$^{-1}$ & $0.40$ & $0.28$--$0.52$ \\
$r$ & discount rate & yr$^{-1}$ & $0.05$ & $0.02$--$0.08$ \\
$p$ & recruitment flux & yr$^{-1}$ & $1$ (norm.) & fixed \\
$l_{50},\delta_m$ & maturity ogive & cm & $45,\,5$ & --- \\
$l_{\mathrm{val}},w_{\mathrm{val}}$ & dome harvest-value centre / width & cm & $52,\,16$ & $30$--$90$ \\
\bottomrule
\end{tabular}\vspace{2pt}
{\footnotesize $^{\dagger}$ In the uncertainty study $L_\infty$ is sampled as
$l_m$ plus a $\pm30\%$ perturbation of the growth gap $L_\infty-l_m$, which keeps
$L_\infty>l_m$ and hence $g>0$ throughout the size range.}
\end{table}

\subsection{The biological mechanism}
\label{subsec:bio_mechanism}

Figure~\ref{fig:mechanism} displays the mechanism at the unfished operating point
($E^*\approx0.63$). As the environmental load rises, the stationary profile
(panel a) does not fall uniformly: near the recruitment size, the density
strictly increases. This occurs because the boundary fixes the entering flux while
crowding-suppressed growth raises the residence density $p/g(E,l_0)$ of recruits.
The sensitivity $\partial_Ex_E$ (panel b) is therefore positive over an
inflow band---about 21 cm wide at the unfished point (spanning from the recruitment size $l_0=10$\,cm to a crossover at $l\approx 31$\,cm) and about 24 cm wide under
the computed canonical harvest---and negative thereafter. Panel (c) resolves this into the
two competing integrals of Section~\ref{sec:stationary}: a positive
residence-time term $\alpha=-\partial_Eg/g$ concentrated near $l_0$, and a
negative cumulative term $\int_{l_0}^lb$ that accrues with size as mortality and
delayed progression remove individuals from the larger classes. Their crossover
is the boundary of the small-size accumulation region.

\begin{figure}[htbp]
\centering
\includegraphics[width=\textwidth]{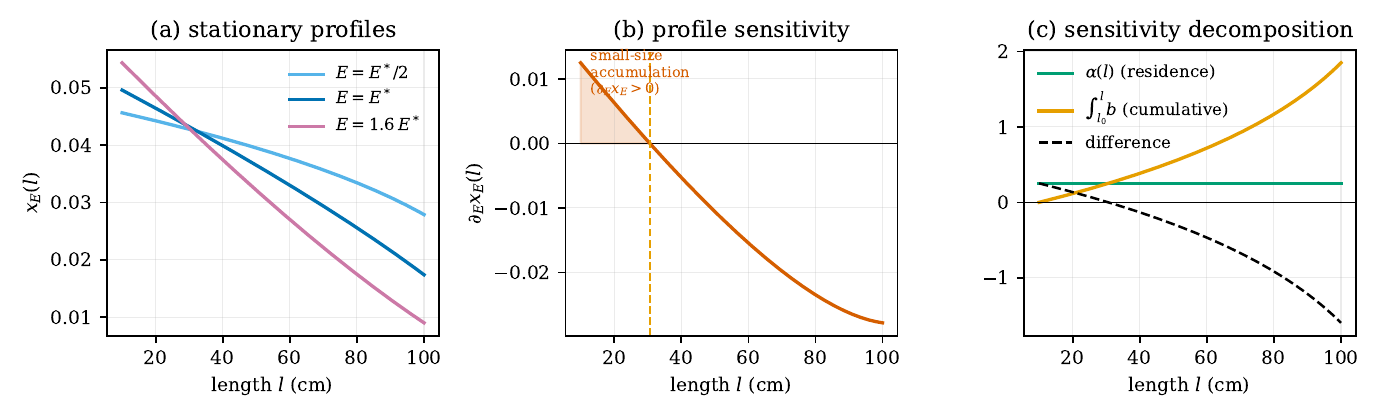}
\caption{Biological mechanism. (a) Stationary profiles $x_E$ at increasing
environmental load; the inflow density rises with $E$. (b) Profile sensitivity
$\partial_Ex_E$, positive over a small-size accumulation band (shaded) and
negative beyond it. (c) Decomposition of
$\partial_E\log x_E$ into the local residence-time gain $\alpha$ and the
cumulative survival/progression loss $\int_{l_0}^lb$.}
\label{fig:mechanism}
\end{figure}

Whether this local accumulation overturns global behaviour is decided by the
integrated balance $\Phi'(E)=\mathsf A(E)-\mathsf C(E)$ of
Theorem~\ref{thm:integrated_monotonicity}. Figure~\ref{fig:closure} shows the
closure map $\Phi$ crossing the identity once and the decomposition
$\mathsf A-\mathsf C$ remaining negative across the relevant range: at baseline
$\mathsf A\approx0.16$ while $\mathsf C\approx0.72$, so $\Phi'(E^*)\approx-0.56$
and the steady state is unique. The accumulation near the inflow is genuine but
the cumulative loss dominates the integral.

\begin{figure}[t]
\centering
\includegraphics[width=0.86\textwidth]{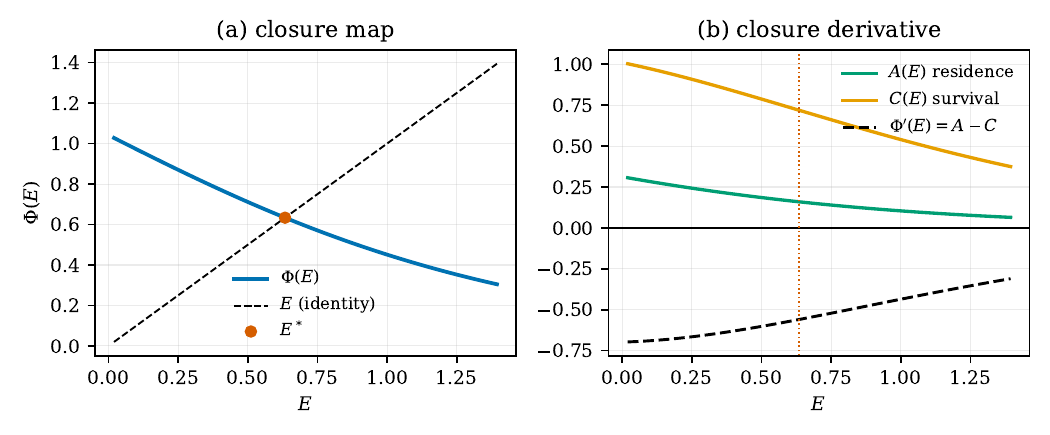}
\caption{Stationary closure. (a) $\Phi(E)$ meets the identity once at $E^*$.
(b) The residence amplification $\mathsf A$, the cumulative reduction
$\mathsf C$, and $\Phi'=\mathsf A-\mathsf C<0$: uniqueness holds because the
cumulative loss exceeds the inflow gain.}
\label{fig:closure}
\end{figure}
\subsection{Where the mechanism matters}
\label{subsec:regime_maps}

Figure~\ref{fig:mechmap} maps the mechanism over the feedback plane
$(\rho_g,\rho_\mu)$. To probe uniqueness as a global property---rather than only
the local slope at the fixed point---panel (a) shows the sufficient-condition
diagnostic $\max_{E\in[0,M_{\mathrm{st}}]}\Phi'(E)$ of
Theorem~\ref{thm:integrated_monotonicity}, evaluated on the stationary closure
interval $[0,M_{\mathrm{st}}]$ with $M_{\mathrm{st}}=\Phi(0)$ recomputed for each
parameter pair: where it is negative, $\Phi'<0$ throughout
$[0,M_{\mathrm{st}}]$ and the steady state is unique. It is negative over all but
$\approx4.4\%$ of the sampled plane, the exception being the extreme
growth-limited corner (large $\rho_g$, negligible $\rho_\mu$) where the residence
amplification $\mathsf A$, although saturating ($\alpha=\rho_g/(1+\rho_gE)$),
is no longer dominated by the cumulative loss across the whole interval. At the
no-feedback corner $\rho_g=\rho_\mu=0$ the rates are $E$-independent and
$\Phi'\equiv0$, so the condition is non-strict there. At the computed fixed point
the local slope $\Phi'(E^*)$ is nonpositive throughout the plane and strictly
negative away from the no-feedback corner. The
width of the small-size accumulation region (panel b), by
contrast, grows steadily with $\rho_g$ and can exceed a third of the size range:
the accumulation phenomenon is widespread even though it does not destabilise the
equilibrium.

\begin{figure}[t]
\centering
\includegraphics[width=0.92\textwidth]{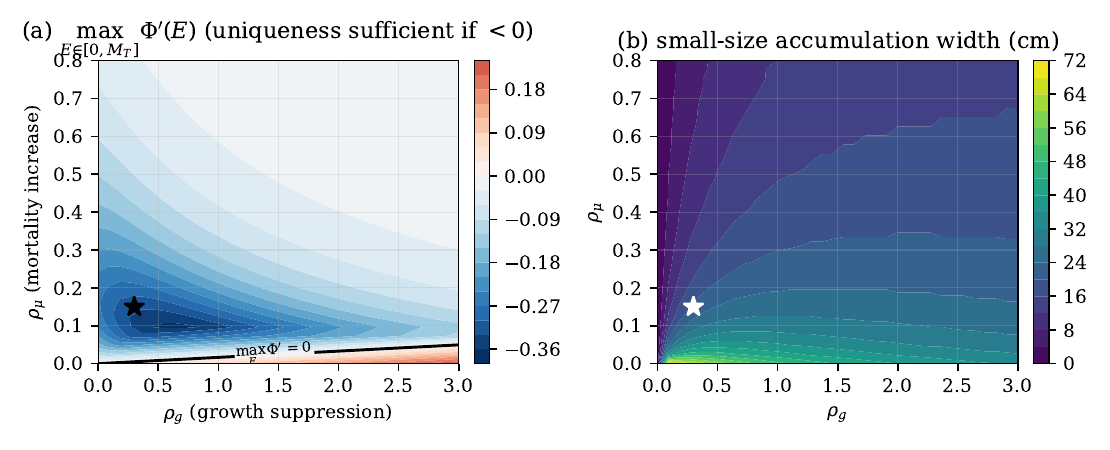}
\caption{Mechanism map over $(\rho_g,\rho_\mu)$ (star: baseline). (a) The
sufficient-uniqueness diagnostic $\max_{E\in[0,M_{\mathrm{st}}]}\Phi'(E)$ on the
stationary closure interval ($M_{\mathrm{st}}=\Phi(0)$ per parameter pair):
negative (hence $\Phi'<0$ throughout $[0,M_{\mathrm{st}}]$ and the equilibrium
unique) over all but the extreme growth-limited corner. (b) Width (cm) of the
small-size accumulation region, which broadens with growth suppression.}
\label{fig:mechmap}
\end{figure}

\subsection{Switching correction and harvest policy}
\label{subsec:policy_numerics}

For monotone value, the reduced switching function is increasing and the rank-one
correction $-\Gamma\psi$ is a modest positive shift ($\Gamma\approx-0.08$,
$1-B\approx1.25$, far from resonance): a single upward crossing survives and the
canonical extremal is a minimum-size policy (Figure~\ref{fig:switching}a; the
maximum-condition residual is zero to machine precision). This is not special to
the baseline: at every resolved point of the feedback plane, policy iteration
returned a verified minimum-size stationary canonical candidate. The picture
changes when the harvest value is dome-shaped
and peaks at small sizes. Protecting small individuals then no longer lets them
grow into a more valuable class, and the canonical switching function acquires
additional crossings (Figure~\ref{fig:switching}b shows a self-consistent
two-crossing extremal, residual zero): the computed stationary canonical
candidate is a harvested window, not a threshold.

\begin{figure}[t]
\centering
\includegraphics[width=0.92\textwidth]{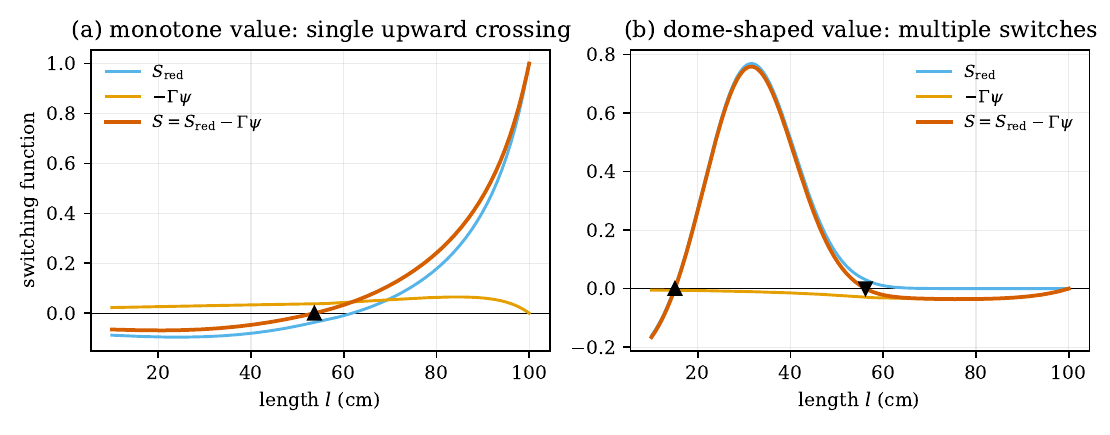}
\caption{Rank-one switching correction $S=S_{\mathrm{red}}-\Gamma\psi$.
(a) Monotone value: a single upward crossing (minimum-size policy).
(b) Dome-shaped value peaked at small sizes: the correction yields multiple
crossings, i.e.\ a harvested size window rather than a threshold (markers: upward
$\blacktriangle$ / downward $\blacktriangledown$ crossings). Both panels are
self-consistent canonical extremals (maximum-condition residual zero).}
\label{fig:switching}
\end{figure}

Figure~\ref{fig:policymap} classifies the canonical extremal across the
dome-value plane $(\rho_g,l_{\mathrm{val}})$ using self-consistent policy
iteration. Minimum-size policies occupy $\approx90.7\%$ of the plane; a band of
multiple-switch (window) policies ($\approx7.7\%$) appears where the value peak
sits at small sizes, and a thin set of cases ($\approx1.7\%$) does not reach a
verified fixed point and is labelled unconverged rather than counted as a
regime. Where the window regime occurs, the computed stationary canonical
candidate is not a minimum-size policy, demonstrating that the canonical
conditions do not generally enforce threshold structure; the transition is the
loss of single crossing flagged by the fine-grid diagnostic of
Section~\ref{sec:stationary_switching} (a grid-level sign-pattern test, not the
interval-validated certificate of Theorem~\ref{thm:validated_certificate}, which
we do not implement numerically here).

\begin{figure}[t]
\centering
\includegraphics[width=0.66\textwidth]{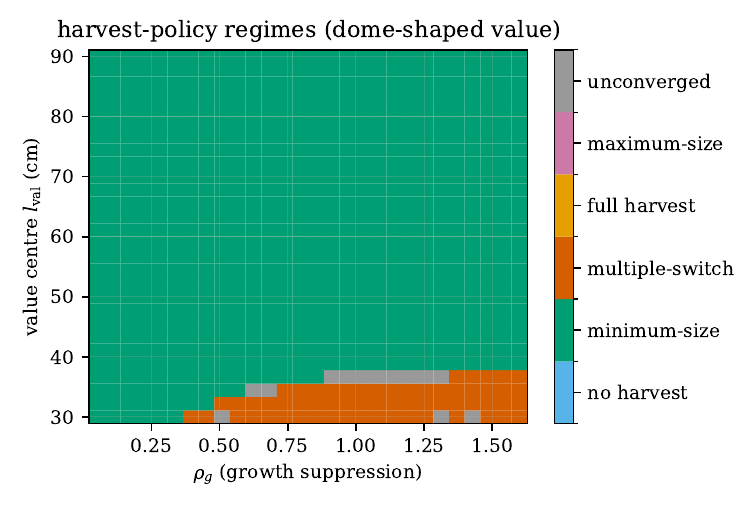}
\caption{Harvest-policy regimes under dome-shaped value over
$(\rho_g,l_{\mathrm{val}})$, from self-consistent policy iteration. Minimum-size
policies dominate; window (multiple-switch) policies appear when the value peak is
at small sizes; unconverged cases are shown separately. Failed computations are
excluded (none occurred), not coded as a regime.}
\label{fig:policymap}
\end{figure}

\subsection{Ecological trade-offs}
\label{subsec:tradeoffs}

Figure~\ref{fig:tradeoffs} traces the consequences of strengthening feedback
(scaling $\rho_g,\rho_\mu$ jointly). Equilibrium yield and biomass fall as
feedback intensifies; the computed canonical threshold $L^*$ shifts downward, so
stronger compensation calls for protecting less of the size range; and the
replacement diagnostics decline. Two replacement quantities must be
distinguished. The equilibrium-adjusted spawning-potential ratio
$\mathrm{SPR}_{\mathrm{eq}}$ stays near $0.5$ across the range (baseline $0.53$);
because its numerator and denominator are evaluated at different equilibrium
environments, this is not the conventional fixed-vital-rate SPR and is not by
itself a sustainability criterion. The invasion replacement quantity, by
contrast, is $\mathcal R_0(u)=\mathcal R(0;u)\approx0.415<1$ at baseline; moreover
the replacement functional satisfies
$\sup_{E\in[0,M_{\mathrm{st}}]}\mathcal R(E;u)=0.415<1$ (the maximum is at $E=0$),
so the harvested stock fails to replace itself at every environment in the
closure interval. The forced model nevertheless maintains a positive stationary
population through the prescribed recruitment flux $p$: a moderate
equilibrium-adjusted SPR coexists with failure of low-density (indeed all-density)
replacement. This is the forced-persistence caveat of
Section~\ref{sec:persistence}, made quantitative. Representative regimes are
summarised in Table~\ref{tab:regimes}.

\begin{figure}[t]
\centering
\includegraphics[width=\textwidth]{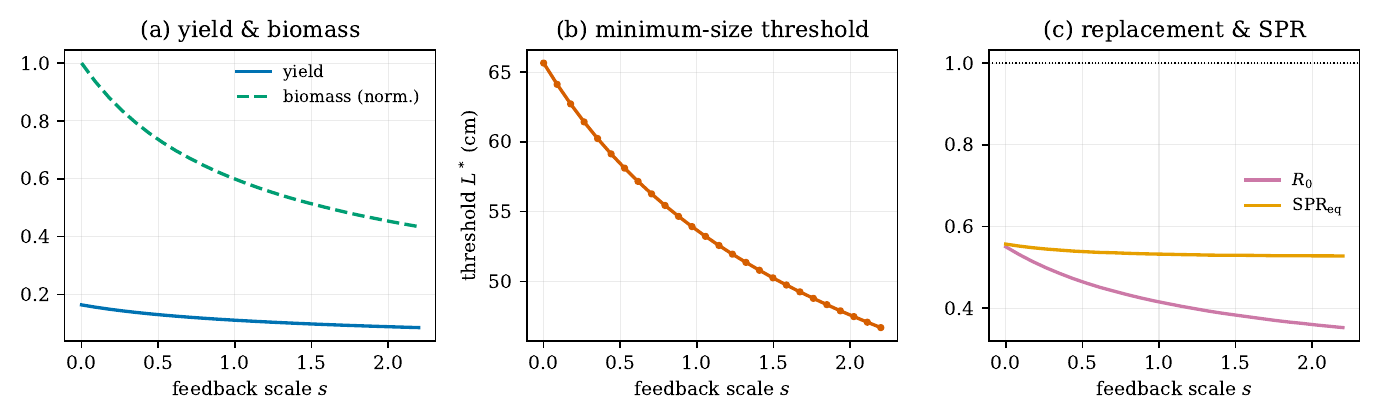}
\caption{Ecological trade-offs versus feedback strength $s$ (joint scaling of
$\rho_g,\rho_\mu$). (a) Yield and (normalised) biomass decline. (b) The computed
canonical threshold $L^*$ falls. (c) The equilibrium-adjusted SPR stays near $0.5$
while the invasion replacement quantity $\mathcal R_0$ remains below one: a moderate equilibrium-adjusted SPR coexists with failure of low-density
replacement.}
\label{fig:tradeoffs}
\end{figure}

\begin{table}[t]
\centering
\footnotesize
\setlength{\tabcolsep}{4.0pt}
\caption{Representative regimes (all computed as self-consistent canonical
extremals by policy iteration). $E^*$, the closure derivative $\Phi'(E^*)$, its
parts $\mathsf A,\mathsf C$, the rank-one scalars $A,B,1-B,\Gamma$, the policy
type, the ecological outputs (yield, biomass, $\mathcal R_0$,
$\mathrm{SPR}_{\mathrm{eq}}$), and the maximum-condition residual ``res''
(mismatch fraction; zero indicates an exact canonical fixed point on the grid).
The two dome rows use the dome payoff $c$ at baseline feedback with the value peak
at an intermediate size (``baseline peak'', $l_{\mathrm{val}}=52$) and at a small
size (``small-size peak'', $l_{\mathrm{val}}=30$, $\rho_g=1$); the former is a
minimum-size candidate, the latter a self-consistent harvested window
(the case of Figure~\ref{fig:switching}b).}
\label{tab:regimes}

\begin{tabular}{lrrrrrrrr}
\toprule
Regime & $E^*$ & $\Phi'$ & $\mathsf A$ & $\mathsf  C$ & $A$ & $B$ & $1{-}B$ & $\Gamma$ \\
\midrule
Baseline & 0.361 & -0.266 & 0.098 & 0.364 & -0.0946 & -0.2470 & 1.247 & -0.0759 \\
Weak feedback & 0.484 & -0.112 & 0.046 & 0.158 & -0.0390 & -0.1047 & 1.105 & -0.0353 \\
Growth-limited & 0.380 & -0.243 & 0.265 & 0.508 & -0.1065 & -0.2423 & 1.242 & -0.0857 \\
Mortality-limited & 0.264 & -0.431 & 0.026 & 0.457 & -0.1465 & -0.3913 & 1.391 & -0.1053 \\
Strong symmetric & 0.240 & -0.496 & 0.161 & 0.657 & -0.1795 & -0.4573 & 1.457 & -0.1232 \\
Dome value, baseline peak & 0.259 & -0.205 & 0.072 & 0.277 & -0.0862 & -0.1899 & 1.190 & -0.0725 \\
Dome value, small-size peak & 0.280 & -0.516 & 0.219 & 0.735 & 0.0343 & -0.4556 & 1.456 & 0.0235 \\
\bottomrule
\end{tabular}

\vspace{1.5em}

\begin{tabular}{llrrrrl}
\toprule
Regime & policy & yield & biom. & $\mathcal R_0$ & SPR & res \\
\midrule
Baseline & min-size & 0.1111 & 0.361 & 0.415 & 0.532 & 0.0e+00 \\
Weak feedback & min-size & 0.1393 & 0.484 & 0.488 & 0.542 & 0.0e+00 \\
Growth-limited & min-size & 0.1088 & 0.380 & 0.426 & 0.480 & 0.0e+00 \\
Mortality-limited & min-size & 0.0871 & 0.264 & 0.360 & 0.564 & 0.0e+00 \\
Strong symmetric & min-size & 0.0792 & 0.240 & 0.338 & 0.528 & 0.0e+00 \\
Dome value, baseline peak & min-size & 0.2692 & 0.259 & 0.272 & 0.369 & 0.0e+00 \\
Dome value, small-size peak & multi-switch & 0.3075 & 0.280 & 0.412 & 0.490 & 0.0e+00 \\
\bottomrule
\end{tabular}
\end{table}

\subsection{Numerical validation and robustness}
\label{subsec:validation}

The exact results provide stringent internal checks. The unification identity
$B|_{r=0}=\Phi'(E^*)$ holds to machine precision (residual $5\times10^{-16}$). The
algebraic reconstruction $\lambda=\lambda_{\mathrm{red}}+\Gamma\psi$ is exact by
construction, so $\|\lambda-\lambda_{\mathrm{red}}-\Gamma\psi\|_\infty$ is zero in
floating point and is not an independent test; we instead report the
$\Gamma$-consistency (closure) residual $|\mathcal G^*[\lambda]-\Gamma|$,
which probes the discrete closure and is $\sim10^{-17}$
(Table~\ref{tab:convergence}). Recomputing the canonical solution at each
resolution, the operating point and rank-one scalars are stable to approximately
$10^{-3}$ by $N=800$. The successive changes are small but non-monotone at finer
resolutions, associated with the grid-dependent placement of the switching point,
so the present experiment does not provide a clean empirical convergence-order
estimate; a dedicated rate study would align the switch with each grid and
integrate separately on either side.

\begin{table}[t]
\centering
\small
\caption{Mesh convergence and residuals (canonical solution recomputed at each
resolution). Columns: grid size $N$; operating point $E^*$; closure derivative
$\Phi'(E^*)$; gain $B$; correction $\Gamma$; successive change
$|E^*_N-E^*_{N/2}|$; observed order; the $\Gamma$-consistency (closure) residual
$|\mathcal G^*[\lambda]-\Gamma|$; and the identity residual $|B(0)-\Phi'(E^*)|$.
The successive change is small ($\lesssim10^{-4}$) but non-monotone, reflecting
grid-dependent placement of the switch; we therefore do not claim a clean
empirical convergence order.}
\label{tab:convergence}
\begin{tabular}{rrrrrrrll}
\toprule
$N$ & $E^*$ & $\Phi'(E^*)$ & $B$ & $\Gamma$ & $|\Delta E^*|$ & order & $\gamma$-res & id-res \\
\midrule
200 & 0.36127 & -0.26620 & -0.24698 & -0.07589 & 2.29e-04 & -- & 0.0e+00 & 0.0e+00 \\
400 & 0.36130 & -0.26623 & -0.24700 & -0.07589 & 2.00e-04 & 1.20 & 1.4e-17 & 0.0e+00 \\
800 & 0.36131 & -0.26624 & -0.24701 & -0.07589 & 1.88e-04 & 1.11 & 2.8e-17 & 1.1e-16 \\
1600 & 0.36132 & -0.26624 & -0.24702 & -0.07589 & 1.82e-04 & -4.96 & 1.4e-17 & 3.3e-16 \\
3200 & 0.36150 & -0.26636 & -0.24713 & -0.07591 & --      & -- & 1.4e-17 & 5.0e-16 \\
\bottomrule
\end{tabular}
\end{table}

Robustness to parameter uncertainty is assessed by Monte-Carlo sampling
($\pm30\%$ uniform on $k,\mu_0,\rho_g,\rho_\mu,u_{\max}$, and on the growth gap
$L_\infty-l_m$ so that $L_\infty>l_m$ and $g>0$ throughout;
Figure~\ref{fig:robust}). Each draw is solved with the window-admitting policy
iteration and its maximum-condition residual is verified. The local closure slope
$\Phi'(E^*)$ is negative in all admissible draws, so the equilibrium is
locally unique throughout the sample; the equilibrium SPR has median $0.53$; and a
standardized-regression sensitivity analysis (linear fit $R^2=0.94$) ranks
$\mu_0$, $\rho_\mu$, and $u_{\max}$ as the dominant controls on the closure
derivative, with $L_\infty$ now minor---the earlier prominence of $L_\infty$ was
an artifact of draws that violated $L_\infty>l_m$. Because the solver admits
harvest windows, the finding that every admissible draw returns a minimum-size
canonical extremal is a genuine result, not an artifact of a threshold-only
search.

\begin{figure}[t]
\centering
\includegraphics[width=\textwidth]{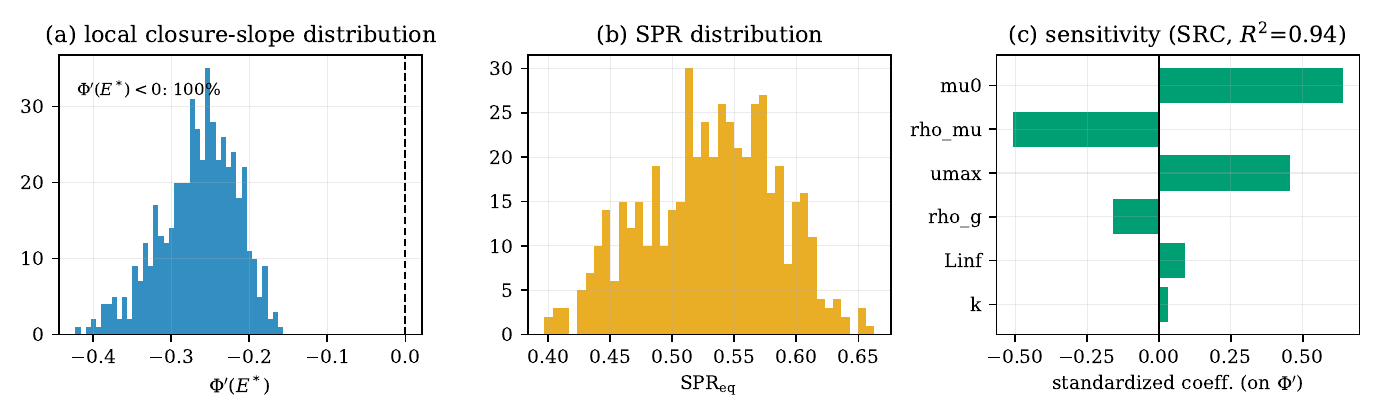}
\caption{Robustness ($\pm30\%$ admissible parameter sampling, $L_\infty>l_m$
enforced). (a) Distribution of the local closure slope $\Phi'(E^*)$: negative in
all admissible draws. (b) Distribution of $\mathrm{SPR}_{\mathrm{eq}}$.
(c) Standardized-regression sensitivity of $\Phi'(E^*)$ to the sampled parameters
(linear fit $R^2=0.94$).}
\label{fig:robust}
\end{figure}

\subsection{Ecological conclusion}
\label{subsec:eco_conclusion}

The analysis identifies a previously hidden ecological trade-off. Crowding-induced
growth suppression increases the residence density of small individuals near the
recruitment size, whereas elevated mortality and delayed progression remove
individuals from the larger size classes. The integrated balance of these two
effects---not the sign of either feedback alone---determines whether the
stationary equilibrium is unique, how far intrinsic replacement falls below the
equilibrium spawning-potential ratio, and whether a conventional minimum-size
harvest rule remains the canonical policy. In the calibrated model the balance is robustly
tipped toward the cumulative loss, so the equilibrium is unique and the
minimum-size rule survives for value that rises with size; yet the same feedback
drives the intrinsic basic reproduction number under the computed canonical harvest below one, so
the exploited stock persists only through recruitment subsidy, and when the
harvest value peaks at small sizes the computed canonical candidate is a harvested window rather
than a threshold. These distinctions are invisible to a pointwise or
fixed-vital-rate analysis and are made quantitative here by the closure derivative
$\Phi'(E^*)=\mathsf A-\mathsf C$ and its exact identity with the discounted
feedback gain.

\appendix

We provide detailed proofs for all the mathematical conclusions in the main text.

\section{Proofs in Section~\ref{sec:stationary}}

\begin{proof}[Proof of Proposition~\ref{prop:frozen_stationary_profile}]
The transformation $y_E=g(E,\cdot)x_E$ reduces the equation to a scalar linear
ODE with unique solution $p\,\sigma_E$. Dividing by $g(E,l)>0$ gives
\eqref{eq:stationary_profile}; since $0<\sigma_E\le1$ and
$g\ge g_{\min}(M_{\mathrm{st}})$ on $[0,M_{\mathrm{st}}]\times\Om$, the bound
follows.
\end{proof}   

\begin{proof}{Proof of Proposition~\ref{prop:Phi_existence}}
    Continuity on $[0,M_{\mathrm{st}}]$ follows from continuity of $g,\mu$ in $E$, the
positive lower bound $g_{\min}(M_{\mathrm{st}})$ of \eqref{eq:H2}, and dominated
convergence. By \eqref{eq:Hst} the continuous $F(E):=\Phi(E)-E$ satisfies
$F(0)=\Phi(0)\ge0$ and $F(M_{\mathrm{st}})=\Phi(M_{\mathrm{st}})-M_{\mathrm{st}}\le0$,
so the intermediate value theorem gives a fixed point. Positivity follows from
$x_E>0$.
\end{proof}

\begin{proof}[Proof of Proposition~\ref{prop:pointwise_failure}]
At $l=l_0$ the cumulative integral in \eqref{eq:dE_log_x} vanishes, giving
$\partial_E\log x_E(l_0)=\alpha(E,l_0)\ge0$; multiplying by $x_E(l_0)>0$ gives
\eqref{eq:inflow_sensitivity}. Strictness and the local statement follow from
$\partial_Eg(E,l_0)<0$ and continuity of the bracket.
\end{proof}

\begin{proof}[Proof of Theorem~\ref{thm:integrated_monotonicity}]
Strict decrease follows from \eqref{eq:Phi_prime_tail}; existence from
Proposition~\ref{prop:Phi_existence}. Two fixed points $E_1<E_2$ would give
$E_1=\Phi(E_1)>\Phi(E_2)=E_2$, a contradiction. Uniqueness of the profile is
Proposition~\ref{prop:frozen_stationary_profile}.
\end{proof}

\section{Proofs in Section~\ref{sec:wellposed}}

\begin{proof}[Proof of Lemma~\ref{lem:green_trace}]
\emph{Amplitude.} The attenuation exponent
$\mathcal A_E(a,t;l)=\exp\!\big(-\!\int_a^t(\partial_lg+\mu+u)\big)$ need not be
$\le1$, since $\partial_lg$ may be negative; using $\mu+u\ge0$ and
$(\partial_lg)\ge-(\partial_lg)^-$,
\begin{equation}\label{eq:attenuation_bound}
0\le\mathcal A_E(a,t;l)
\le\exp\!\Big(\int_a^t\|(\partial_lg(E(s),\cdot))^-\|_{L^\infty}\dd s\Big)
\le e^{T\|\partial_lg\|_{L^\infty}},
\end{equation}
so $x$ is bounded by $\max(\|\phi\|_\infty,\sup_s p(s)/g_{\min})\,
e^{T\|\partial_lg\|_\infty}=:C_T^0$; all amplitude estimates hold with this larger
constant.

\emph{Existence of the one-sided trace.} Boundedness alone does not give a trace;
we argue from the representation \eqref{eq:frozen_representation}. As
$l\uparrow l_m$, the backward characteristic footpoint $\eta_E(0;t,l)$ and the
entrance time $\tau_E(t,l)$ have one-sided limits (each is monotone in $l$);
$\phi\in BV$ and $p\in BV$ have one-sided limits at those footpoints; for a.e.\
$s$, $u(s,\cdot)\in BV(\Om)$ has a one-sided trace along the characteristic; and
the attenuation integral converges by dominated convergence using
\eqref{eq:attenuation_bound}. Hence $x(t,l_m^-)=\lim_{l\uparrow l_m}x(t,l)$ exists
for a.e.\ $t$ and is bounded; multiplying by $g\le g_{\max}$ gives
$\beta\in L^\infty(0,T)$.

\emph{Green identity.} Multiply \eqref{eq:state} by $\psi\in W^{1,\infty}(\Om)$,
integrate over $(t,t+h)\times\Om$, and integrate the flux term by parts in $l$;
the boundary contributions are the prescribed inflow flux
$g(E,l_0)x(\cdot,l_0)=p$ at $l_0$ and the outflow flux $\beta$ at $l_m$ (which
exists by the previous step). The manipulation does not require $\psi(l_m)=0$,
giving \eqref{eq:green_time_identity}.
\end{proof}

\begin{proof}[Proof of Lemma~\ref{lem:frozen}]
Existence, uniqueness, and nonnegativity follow from
\eqref{eq:frozen_representation}. The $L^1$ bound is the population estimate; the
$L^\infty$ bound uses the bounded data, $g\ge g_{\min}$, and the attenuation bound
\eqref{eq:attenuation_bound} (which is $\le e^{T\|\partial_lg\|_\infty}$, not
necessarily $\le1$). The outflow trace bound is Lemma~\ref{lem:green_trace}. None
of these uses a variation bound on $E$.
\end{proof}

\begin{proof}[Proof of Lemma~\ref{lem:env_lipschitz}]
Apply the Green identity \eqref{eq:green_time_identity} with $\psi=\chi\in
W^{1,\infty}(\Om)$ on $(t,t+h)$, divide by $h$, and let $h\downarrow0$: the
left side tends to $(\mathcal FE)'(t)$ at Lebesgue points, the interior terms to
$\int_\Om(g x_E\chi'-(\mu+u)x_E\chi)$, the inflow term to $\chi(l_0)p(t)$, and the
outflow term to $-\chi(l_m)g(E,l_m)x_E(t,l_m^-)$, giving
\eqref{eq:E_prime_identity}. The bound \eqref{eq:E_prime_bound} follows from
$|p|\le\|p\|_\infty$, the trace bound $g(E,l_m)x_E(\cdot,l_m^-)\le C_T^0$, and the
$L^1$ and pointwise bounds on $x_E$, all from \eqref{eq:frozen_uniform_bounds};
every constant is independent of $E$. Integrating gives the Lipschitz bound.
\end{proof}

\begin{proof}[Proof of Lemma~\ref{lem:frozen_bv}]
Fix $t$. For $s\le t$ the map $l\mapsto\eta_E(s;t,l)$ is monotone (characteristics
do not cross), as are the footpoint and entrance-time maps; composition with a
monotone map does not increase total variation, so
$\TV_l\big(f(s,\eta_E(s;t,\cdot))\big)\le\TV_\Om(f(s,\cdot))$ for any
$f(s,\cdot)\in BV(\Om)$, with $f=\partial_lg+\mu+u$ giving
$\TV_l\le\|\partial_lg\|_{BV}+\|\mu\|_{BV}+C_u$.

The attenuation exponent on the boundary-fed branch is
$H(l)=\int_{\tau_E(t,l)}^t f(s,\eta_E(s;t,l))\dd s$, whose lower limit moves with
$l$. We prove the key estimate
\begin{equation}\label{eq:moving_limit_TV}
\TV_l H
\le\int_0^t\TV_l\big(f(s,\eta_E(s;t,\cdot))\big)\dd s
+\|f\|_{L^\infty}\,\TV_l\big(\tau_E(t,\cdot)\big).
\end{equation}
For a partition $l_0\le a_0<\cdots<a_k\le l_m$, write each increment as
\[
H(a_{j})-H(a_{j-1})
=\underbrace{\int_{\tau_E(t,a_{j-1})}^{t}\!\big[f(s,\eta_E(s;t,a_{j}))-f(s,\eta_E(s;t,a_{j-1}))\big]\dd s}_{\text{(I): common time interval}}
\;-\;\underbrace{\int_{\tau_E(t,a_{j-1})}^{\tau_E(t,a_{j})}\! f(s,\eta_E(s;t,a_{j}))\dd s}_{\text{(II): differing lower limits}}.
\]
Summing $|\text{(I)}|$ over $j$ and applying Fubini gives at most
$\int_0^t\TV_l(f(s,\eta_E(s;t,\cdot)))\dd s$; summing $|\text{(II)}|$ gives at most
$\|f\|_{L^\infty}\sum_j|\tau_E(t,a_j)-\tau_E(t,a_{j-1})|\le\|f\|_{L^\infty}\TV_l(\tau_E(t,\cdot))$.
Taking the supremum over partitions yields \eqref{eq:moving_limit_TV}. The first
term is bounded by $T(\|\partial_lg\|_{BV}+\|\mu\|_{BV}+C_u)$ as above. For the
second, $\tau_E(t,\cdot)$ is monotone in $l$ with values in $[0,t]$, so
$\TV_l(\tau_E(t,\cdot))\le t\le T$ directly from monotonicity (the maximum possible
travel time), with no Lipschitz claim required. Hence
$\TV_l H\le C(T,C_u)$, and $\TV(\mathcal A_E)\le\|\mathcal A_E\|_\infty\TV_l H$
since $h\mapsto e^{-h}$ is Lipschitz on bounded sets and $\mathcal A_E$ is bounded
by \eqref{eq:attenuation_bound}.

For the boundary density, the $L_E$-Lipschitz bound on $E$ makes
$t\mapsto1/g(E(t),l_0)$ Lipschitz, hence of bounded variation, with
$\TV_{(0,T)}\big(1/g(E(\cdot),l_0)\big)\le g_{\min}^{-2}\|\partial_Eg\|_\infty L_E
T$; combined with $p\in BV(0,T)$ and composition with the monotone map
$\tau_E(t,\cdot)$ (composition with a monotone map does not increase total
variation), the boundary factor $p(\tau_E)/g(E(\tau_E),l_0)$ has uniformly bounded
spatial variation. The $BV$ product rule, the $BV$ datum $\phi$, and the single
interface between initial-fed and boundary-fed characteristics (one jump of
uniformly bounded size) then yield \eqref{eq:frozen_bv_bound}.
\end{proof}

\begin{proof}[Proof of Lemma~\ref{lem:short_stability}]
Lipschitz dependence of $g$ on $E$ gives, for the flows and entrance-time maps,
$\sup_{s\le t\le\tau}\sup_l|\eta_{E_1}-\eta_{E_2}|\le C_T\tau\|E_1-E_2\|_{C}$.
Applying the translation estimate $\|f(\cdot+h)-f\|_{L^1}\le|h|\TV(f)$ to the
factors of \eqref{eq:frozen_representation} and using the uniform spatial-$BV$
bound of Lemma~\ref{lem:frozen_bv} (available because
$E_1,E_2\in\mathcal B_{M_T}^{L_E}$) together with Lemma~\ref{lem:frozen} yields
\eqref{eq:short_stability}.
\end{proof}

\begin{proof}[Proof of Theorem~\ref{thm:wellposed}]
The contraction above on the closed set $\mathcal B_{M_T}^{L_E}$ gives a unique
fixed point $E^u\in\mathcal B_{M_T}^{L_E}$ on a short interval; the population
bound \eqref{eq:Ebound} prevents escape from $\mathcal B_{M_T}$ and the uniform
Lipschitz bound \eqref{eq:E_prime_bound} prevents escape from
$\mathcal B_{M_T}^{L_E}$, so the construction iterates over $[0,T]$ with constants
independent of the starting time. The amplitude bounds are
Lemma~\ref{lem:frozen}; the spatial-$BV$ bound is Lemma~\ref{lem:frozen_bv},
which applies precisely because $E^u$ is $L_E$-Lipschitz.
\end{proof}

\begin{proof}[Proof of Proposition~\ref{prop:continuous_dependence}]
Work on a short interval $[0,\tau]$ and iterate. Subtract the characteristic
representations \eqref{eq:frozen_representation} of $x_1,x_2$ along their
respective flows $\eta_1,\eta_2$ and reaction factors $\mathcal A_1,\mathcal A_2$.
The control enters through the attenuation exponent evaluated along the
characteristic; the essential point is that $u_1,u_2$ are evaluated along different characteristics, so we split
\[
u_1(s,\eta_1(s))-u_2(s,\eta_2(s))
=\underbrace{\big[u_1(s,\eta_1(s))-u_1(s,\eta_2(s))\big]}_{\text{(a)}}
+\underbrace{\big[u_1(s,\eta_2(s))-u_2(s,\eta_2(s))\big]}_{\text{(b)}}.
\]
Term (a) is a spatial translation of the fixed $BV$ function $u_1(s,\cdot)$ by
$|\eta_1(s)-\eta_2(s)|$; the translation estimate
$\|w(\cdot+\delta)-w\|_{L^1}\le|\delta|\TV(w)$ and the uniform spatial-$BV$ bound
$\TV_\Om(u_1(s,\cdot))\le C_u$ bound its $L^1$-in-$l$ contribution by
$C_u\sup_s|\eta_1(s)-\eta_2(s)|$. Since the flows depend on $E_i$ through $g$,
$\sup_{s\le\tau}\sup_l|\eta_1-\eta_2|\le C_T\tau\,\|E_1-E_2\|_{C}$, and
$|E_1-E_2|\le\|\chi\|_\infty\|x_1-x_2\|_{L^1}$. Term (b) is evaluated along the
common characteristic $\eta_2$; a change of variables back to $l$ bounds its
contribution by $C_T\|u_1-u_2\|_{L^1((0,\tau)\times\Om)}$. Here the relevant
change-of-variables Jacobian is $\partial\eta_2(s;l)/\partial l$, the derivative of
the flow with respect to its initial size; by the variational equation
$\tfrac{d}{ds}(\partial_l\eta_2)=\partial_lg(E_2,\eta_2)\,\partial_l\eta_2$ it
satisfies
$\exp(-\!\int_0^s\|\partial_lg\|_\infty)\le\partial_l\eta_2\le
\exp(\int_0^s\|\partial_lg\|_\infty)$, so it is bounded above and below by the
local bound on $\partial_lg$ from \eqref{eq:H5} (not merely by the range of $g$). The same split applied
to the boundary factor $p_i(\tau_{E_i})/g(E_i,l_0)$ gives, by (a)-type translation
in the entrance time plus (b)-type direct difference, a bound
$C_T(\|p_1-p_2\|_{L^1(0,\tau)}+\sup_s|\tau_{E_1}-\tau_{E_2}|)$ with
$\sup_s|\tau_{E_1}-\tau_{E_2}|\le C_T\tau\|E_1-E_2\|_C$; the initial-datum factor
contributes $\|\phi_1-\phi_2\|_{L^1}$ directly. Collecting terms,
\[
\|x_1-x_2\|_{C([0,\tau];L^1)}
\le C_T\big(\|\phi_1-\phi_2\|_{L^1}+\|p_1-p_2\|_{L^1}+\|u_1-u_2\|_{L^1}\big)
+C_T\tau\,\|x_1-x_2\|_{C([0,\tau];L^1)};
\]
choosing $\tau$ with $C_T\tau\le\tfrac12$ absorbs the last term, and iterating over
$\lceil T/\tau\rceil$ intervals (with constants uniform under the common
$L^\infty$ and spatial-$BV$ bounds) gives \eqref{eq:continuous_dependence}.
\end{proof}

\section{Proofs of Section~\ref{sec:persistence}}

\begin{proof}[Proof of Proposition~\ref{prop:R0_monotone}]
A stationary renewal profile is proportional to the unit-flux profile,
$x=b\,\hat x_E$; the renewal boundary condition \eqref{eq:renewal_bc} is
homogeneous of degree one in $x$ and reduces to $\mathsf A(E;u)=1$, independent of
$b$. The closure $E=\int\chi x=b\int\chi\hat x_E$ then fixes the amplitude $b$ as in
\eqref{eq:renewal_equilibrium}, which is positive and finite when $E^{\dagger}>0$
and $\int\chi\hat x_{E^{\dagger}}>0$. At $E=0$ a positive state would force
$\int\chi x>0$, contradicting $E=0$ when $\chi>0$ on a positive-measure set.
Monotonicity in $u$ follows since increasing $u$ lowers
the survival factor; strict monotonicity in $E$ and the endpoint signs give a
unique root by the intermediate value theorem.
\end{proof}

\section{Proofs in Section~\ref{sec:control_existence}}

\begin{proof}[Proof of Lemma~\ref{lem:UadBV_compact}]
Convexity follows from convexity of the box constraint and subadditivity of total
variation. By Banach--Alaoglu a bounded sequence has a weak-$\ast$ limit $u$ with
$0\le u\le u_{\max}$; passing to the limit in \eqref{eq:TV_distributional} keeps
the variation bound, so $u\in\UadBV$.
\end{proof}

\begin{proof}[Proof of Lemma~\ref{lem:weak_time_modulus}]
Let $\psi\in W^{1,\infty}(\Om)$ with $\|\psi\|_{W^{1,\infty}}\le1$. The Green
identity \eqref{eq:green_time_identity} of Lemma~\ref{lem:green_trace}---valid for
every $\psi\in W^{1,\infty}(\Om)$, not only those vanishing at $l_m$---gives
\begin{align*}
\Big|\int_\Om\big(x^u(t+h,\cdot)-x^u(t,\cdot)\big)\psi\Big|
&\le
\int_t^{t+h}\!\!\int_\Om g(E^u,l)x^u|\partial_l\psi|\dd l\dd s
\nonumber\\
&\quad+\int_t^{t+h}\!\!\int_\Om\big(\mu(E^u,l)+u\big)x^u|\psi|\dd l\dd s
\nonumber\\
&\quad+\int_t^{t+h} p(s)|\psi(l_0)|\dd s
+\int_t^{t+h} g(E^u(s),l_m)x^u(s,l_m^-)|\psi(l_m)|\dd s.
\end{align*}
The first three terms are bounded by $C_T h$ using the uniform $L^1$ state bound
and boundedness of $g,\mu,u,p$. For the outflow term, Lemma~\ref{lem:green_trace}
gives the trace bound $g(E^u(s),l_m)x^u(s,l_m^-)\le g_{\max}C_T$ in
$L^\infty(0,T)$, so it too is bounded by $C_Th$. Taking the supremum over $\psi$
gives \eqref{eq:weak_time_estimate}.
\end{proof}

\begin{proof}[Proof of Lemma~\ref{lem:strong_time_modulus}]
Apply \eqref{eq:BV_interpolation} to $v=x^u(t+h,\cdot)-x^u(t,\cdot)$; the weak norm
is bounded by Lemma~\ref{lem:weak_time_modulus} and the $L^1+\TV$ factor by
\eqref{eq:uniform_state_BV}.
\end{proof}

\begin{proof}[Proof of Proposition~\ref{prop:strong_state_compactness}]
For each $t$, \eqref{eq:uniform_state_BV} and the compact embedding
$BV(\Om)\hookrightarrow\hookrightarrow L^1(\Om)$ give pointwise relative
compactness; Lemma~\ref{lem:strong_time_modulus} gives equicontinuity of
$t\mapsto x^u(t,\cdot)$ into $L^1(\Om)$. The vector-valued Arzel\`a--Ascoli theorem
yields relative compactness in $C([0,T];L^1(\Om))$.
\end{proof}

\begin{proof}[Proof of Proposition~\ref{prop:control_state_closed}]
By Proposition~\ref{prop:strong_state_compactness}, $x^{u_n}\to\bar x$ strongly,
hence
\[
E^{u_n}(t)=\int_\Om\chi x^{u_n}(t,\cdot)
\longrightarrow
\bar E(t):=\int_\Om\chi\bar x(t,\cdot)
\]
uniformly on $[0,T]$, so $g(E^{u_n},\cdot)\to g(\bar E,\cdot)$ and
$\mu(E^{u_n},\cdot)\to\mu(\bar E,\cdot)$ uniformly. For the bilinear term, write,
for bounded $\psi$,
\[
\int_0^T\!\!\int_\Om u_n x^{u_n}\psi
=
\int_0^T\!\!\int_\Om u_n\bar x\,\psi
+\int_0^T\!\!\int_\Om u_n(x^{u_n}-\bar x)\psi.
\]
Since $\bar x\psi\in L^1$, weak-$\ast$ convergence handles the first term; the
second is bounded by $u_{\max}\|\psi\|_\infty\|x^{u_n}-\bar x\|_{L^1}\to0$. Passing
to the limit in the weak formulation shows $\bar x$ solves the state equation with
control $u$, so $\bar x=x^u$ by uniqueness; the limit being independent of the
subsequence gives the claim.
\end{proof}

\begin{proof}[Proof of Theorem~\ref{thm:optimal_existence_BV}]
Let $(u_n)\subset\UadBV$ be maximizing. By Lemma~\ref{lem:UadBV_compact},
$u_n\overset{\ast}{\rightharpoonup}u_T^*\in\UadBV$; by
Proposition~\ref{prop:control_state_closed}, $x^{u_n}\to x^{u_T^*}$ strongly in
$C([0,T];L^1)$. Writing
\[
\int_0^T\!\!\int_\Om e^{-rt}c\,u_n x^{u_n}
=\int_0^T\!\!\int_\Om e^{-rt}c\,u_n x^{u_T^*}
+\int_0^T\!\!\int_\Om e^{-rt}c\,u_n(x^{u_n}-x^{u_T^*}),
\]
the first term converges by weak-$\ast$ convergence (since
$e^{-rt}c\,x^{u_T^*}\in L^1$) and the second to zero by strong $L^1$ convergence,
so $J_T(u_n)\to J_T(u_T^*)=V_T^{BV}$.
\end{proof}

\begin{proof}[Proof of Proposition~\ref{prop:threshold_existence}]
A maximizing sequence $L_n$ has, by $BV(0,T)\hookrightarrow\hookrightarrow
L^1(0,T)$, a subsequence $L_n\to L_T^*\in\mathcal L_{ad}$ in $L^1$, whence
$u_{L_n}\to u_{L_T^*}$ in $L^1$; continuous dependence and boundedness give
$J_T(u_{L_n})\to J_T(u_{L_T^*})$.
\end{proof}

\section{Proofs in Section~\ref{sec:adjoint}}

\begin{proof}[Proof of Proposition~\ref{prop:state_differentiability}]
Write $x^\varepsilon:=x^{u^*+\varepsilon h}$, $E^\varepsilon$ for its environment,
and $w^\varepsilon:=(x^\varepsilon-x^*)/\varepsilon-z$ for the remainder we must
show is $o(1)$ in $C([0,T];L^1)$. Continuous dependence
(Proposition~\ref{prop:continuous_dependence}) gives the uniform difference-quotient
bounds $\|x^\varepsilon-x^*\|_{C([0,T];L^1)}=O(\varepsilon)$ and
$\|E^\varepsilon-E^*\|_{C([0,T])}=O(\varepsilon)$, so
$q^\varepsilon:=(x^\varepsilon-x^*)/\varepsilon$ is bounded in $C([0,T];L^1)$ and
$\delta E^\varepsilon:=(E^\varepsilon-E^*)/\varepsilon$ is bounded in $C([0,T])$,
with $|\delta E^\varepsilon|\le\|\chi\|_\infty\|q^\varepsilon\|_{L^1}$.

\emph{Expansions.} By $E\mapsto g(E,\cdot)\in C^1(\R_+;W^{1,\infty})$ and
$E\mapsto\mu(E,\cdot)\in C^1(\R_+;L^\infty)$,
\[
g(E^\varepsilon,\cdot)-g(E^*,\cdot)=\partial_Eg(E^*,\cdot)\,(E^\varepsilon-E^*)+R_g^\varepsilon,
\quad
\mu(E^\varepsilon,\cdot)-\mu(E^*,\cdot)=\partial_E\mu(E^*,\cdot)\,(E^\varepsilon-E^*)+R_\mu^\varepsilon,
\]
with $\|R_g^\varepsilon\|_{W^{1,\infty}}+\|R_\mu^\varepsilon\|_{L^\infty}
=o(\|E^\varepsilon-E^*\|_C)=o(\varepsilon)$.

\emph{Equation for the remainder.} Subtracting the transport equations for
$x^\varepsilon$ and $x^*$, dividing by $\varepsilon$, and subtracting the
linearized equation \eqref{eq:linearized_state}, the remainder $w^\varepsilon$
solves a transport equation with the same principal part as the linearization,
\[
\partial_t w^\varepsilon+\partial_l\!\big(g(E^*,l)w^\varepsilon\big)
+\big(\mu(E^*,l)+u^*\big)w^\varepsilon
=\partial_l\!\big(\partial_Eg(E^*,l)\,\delta E^\varepsilon\,(x^\varepsilon-x^*)\big)
-\partial_E\mu\,\delta E^\varepsilon (x^\varepsilon-x^*)+\rho^\varepsilon,
\]
where every term on the right is either (i) a product of a bounded factor with one
of $x^\varepsilon-x^*=O(\varepsilon)$ or $\delta E^\varepsilon-\delta E$, or (ii) a
remainder $\rho^\varepsilon$ collecting $R_g^\varepsilon,R_\mu^\varepsilon$ tested
against the bounded $q^\varepsilon$; hence
$\|\rho^\varepsilon\|_{L^1}=o(1)$ and the quadratic products are $O(\varepsilon)$.
The boundary trace is control-independent, so $w^\varepsilon(t,l_0)$ satisfies the
homogeneous version of \eqref{eq:linearized_boundary} up to an $o(1)$ term obtained
by the same expansion of $g(E^\varepsilon,l_0)$, and $w^\varepsilon(0,\cdot)=0$.

\emph{Grönwall closure.} Testing this transport equation by $\sgn(w^\varepsilon)$
and integrating in $l$ (the outflow trace at $l_m$ has a sign that removes mass,
and the inflow trace contributes the $o(1)$ boundary term) gives
\[
\frac{d}{dt}\|w^\varepsilon(t,\cdot)\|_{L^1}
\le C_T\,\|w^\varepsilon(t,\cdot)\|_{L^1}+\beta^\varepsilon(t),
\qquad \|\beta^\varepsilon\|_{L^1(0,T)}=o(1)+O(\varepsilon),
\]
with $C_T$ the uniform coefficient bound from \eqref{eq:H2}--\eqref{eq:H5} on
$[0,M_T]$ (the $\partial_l(g\,\cdot)$ term contributes through
$\|\partial_lg\|_\infty$ exactly as in
Proposition~\ref{prop:continuous_dependence}). Since $w^\varepsilon(0,\cdot)=0$,
Grönwall gives $\|w^\varepsilon\|_{C([0,T];L^1)}\le e^{C_TT}\|\beta^\varepsilon\|_{L^1}\to0$.
Thus $q^\varepsilon\to z$ in $C([0,T];L^1)$, which is the claim; uniqueness of $z$
follows from the same Grönwall estimate applied to a difference of two solutions.
\end{proof}

\begin{proof}[Proof of Lemma~\ref{lem:Gamma_bounded}]
From \eqref{eq:Gamma_BV},
\[
|\mathcal G_{x^*}[\lambda](t)|
\le\|\lambda(t,\cdot)\|_\infty\Big[
\|\partial_Eg\|_\infty|x^*(t,l_0^+)|
+\|\partial_Eg\|_\infty\TV_\Om(x^*(t,\cdot))
+(\|\partial_l\partial_Eg\|_\infty+\|\partial_E\mu\|_\infty)\|x^*(t,\cdot)\|_{L^1}\Big],
\]
and the bracket is uniformly bounded by the state bounds and
\eqref{eq:adjoint_assumptions}.
\end{proof}

\begin{proof}[Proof of Proposition~\ref{prop:gradient_formula}]
Multiply \eqref{eq:linearized_state} by $\lambda_{\mathrm{pv}}^*=e^{-rt}\lambda^*$,
integrate over $(0,T)\times\Om$, and integrate by parts in $t$ and $l$. The
terminal, outflow, fixed-initial, and linearized-inflow conditions remove the
boundary variations; the environmental-variation terms combine into
$\mathcal G_{x^*}[\lambda^*]$, and \eqref{eq:adjoint_BV} cancels every term in $z$,
leaving $\int_0^Te^{-rt}\int_\Om x^*(c-\lambda^*)h$.
\end{proof}

\begin{proof}[Proof of Theorem~\ref{thm:variational_inequality}]
Convexity of $\UadBV$ makes $u_\varepsilon=u^*+\varepsilon(v-u^*)$ admissible for
$\varepsilon\in[0,1]$; optimality gives $J_T(u_\varepsilon)-J_T(u^*)\le0$. Divide
by $\varepsilon$, let $\varepsilon\downarrow0$, and apply
Proposition~\ref{prop:gradient_formula}.
\end{proof}

\begin{proof}[Proof of Theorem~\ref{thm:bang_bang}]
Under \eqref{eq:inactive_BV}, small one-sided $BV$ bump directions into the box,
supported in small neighborhoods, are feasible; localizing
\eqref{eq:variational_inequality} gives
$x^*S^*(v-u^*)\le0$ for all $v\in[0,u_{\max}]$ at a.e.\ Lebesgue point, equivalent
to the stated arg-max; for $x^*>0$, maximizing the affine map gives
\eqref{eq:bang_bang_rule}.
\end{proof}

\section{Proofs in Section~\ref{sec:threshold}}

\begin{proof}[Proof of Lemma~\ref{lem:sign_decomposition}]
Since $\{S^*<0\}$ is a lower set it equals $[l_0,L_-)$ and $\{S^*>0\}$ is an upper
set equal to $(L_+,l_m]$; the ordered sign condition forbids a negative value above
a positive one, so $L_-\le L_+$, and on the open gap $(L_-,L_+)$ the sign is
neither $<0$ nor $>0$, hence $S^*=0$ there, giving \eqref{eq:sign_decomp}. If both
sign sets are nonempty, $L_-,L_+\in(l_0,l_m)$ and continuity forces
$S^*(L_-)=S^*(L_+)=0$ as limits of negative (resp.\ positive) values. If, say,
$\{S^*<0\}=\varnothing$ then $L_-=l_0$ by convention and no zero is forced at
$l_0$; symmetrically when $\{S^*>0\}=\varnothing$. The converse is immediate.
\end{proof}

\begin{proof}[Proof of Theorem~\ref{thm:weak_threshold_structure}]
Lemma~\ref{lem:sign_decomposition} gives the sign pattern in each case; combined
with the bang--bang rule \eqref{eq:bang_bang_rule} (valid where $x^*>0$) it yields
$u^*=0$ where $S^*<0$ and $u^*=u_{\max}$ where $S^*>0$. In cases (i)--(ii) one sign
set is empty, so no interior threshold is forced and the conclusion holds off the
zero set; in case (iii) both endpoints are interior zeros and the open interval
$(L_-,L_+)$ is the singular set.
\end{proof}

\begin{proof}[Proof of Theorem~\ref{thm:strict_threshold_structure}]
The cases are Definition~\ref{def:strict_single_crossing} combined with
\eqref{eq:bang_bang_rule}.
\end{proof}

\begin{proof}[Proof of Proposition~\ref{prop:threshold_measurable}]
For $a\in\Om$, $\{t:L_+(t)<a\}=\bigcup_{q\in\mathbb Q,\,l_0<q<a}\{t:S^*(t,q)>0\}$
is measurable since each $S^*(\cdot,q)$ is; similarly for $L_-$.
\end{proof}

\begin{proof}[Proof of Proposition~\ref{prop:threshold_regular}]
The implicit function theorem.
\end{proof}

\begin{proof}{Proof of Proposition~\ref{prop:single_crossing_weaker}}
Strict increase gives at most one zero with the correct orientation. For the
converse, with $\theta(l)=(l-l_0)/(l_m-l_0)$ and $A>2$, set
$\displaystyle S_A(l):=\big(\theta-\tfrac12\big)\big(1+A\theta(1-\theta)\big)$.
Since $1+A\theta(1-\theta)>0$, $\operatorname{sgn}S_A=\operatorname{sgn}(\theta-\tfrac12)$, a unique
upward crossing at $L=(l_0+l_m)/2$; yet $\dd S_A/\dd\theta|_0=1-A/2<0$, so $S_A$ is
not increasing.
\end{proof}

\begin{proof}[Proof of Proposition~\ref{prop:quantitative_single_crossing}]
The sign margins exclude exterior zeros; the slope bound makes $S$ strictly
increasing on the central interval, and the endpoint signs give one zero by the
intermediate value theorem. The perturbation bounds preserve both.
\end{proof}

\section{Proofs in Section~\ref{sec:stationary_switching}}

\begin{proof}[Proof of Theorem~\ref{thm:rank_one_correction}]
The difference $\delta:=\lambda-\lambda_{\mathrm{red}}$ solves
$\bar g\delta'=a\delta-\chi\Gamma$, $\delta(l_m)=0$, which by linearity and
\eqref{eq:auxiliary_psi} equals $\Gamma\psi$. Applying $\mathcal G^*$,
$\Gamma=\mathcal G^*[\lambda]=A+\Gamma B$, so $(1-B)\Gamma=A$; under $1-B\neq0$,
$\Gamma=A/(1-B)$. Substitution gives \eqref{eq:rank_one_switching}.
\end{proof}

\begin{proof}[Proof of Lemma~\ref{lem:feedback_vanishes}]
$A,B$ are linear in $g_E,g_E',\mu_E$; the uniform $BV/L^1$ bounds on $x^*$ and
uniform bounds on $\lambda_{\mathrm{red}},\psi$ give $|A|+|B|\le C\varepsilon$;
$|1-B|\ge1-C\varepsilon$ closes the estimate.
\end{proof}

\begin{proof}[Proof of Theorem~\ref{thm:small_feedback_monotone}]
$|\Gamma|\le|A|/(1-\beta)$, so
$S'\ge\eta-\frac{|A|}{1-\beta}\|\psi'\|_\infty>0$; the endpoint signs give one
crossing.
\end{proof}

\begin{proof}[Proof of Theorem~\ref{thm:small_feedback_single_crossing}]
The value bound gives $S<0$ on the left margin and $S>0$ on the right margin, so
all zeros lie centrally; the slope bound gives $S'>0$ there; the opposite end
signs give one crossing. This is Proposition~\ref{prop:quantitative_single_crossing}
applied to $S=S_{\mathrm{red}}-\Gamma\psi$.
\end{proof}

\begin{proof}[Proof of Proposition~\ref{prop:threshold_displacement}]
$0=S(L^*)=S_{\mathrm{red}}(L^*)-\Gamma\psi(L^*)$; expanding about
$L_{\mathrm{red}}$ with $S_{\mathrm{red}}(L_{\mathrm{red}})=0$ and
$L^*-L_{\mathrm{red}}=O(\Gamma)$ gives the claim.
\end{proof}

\begin{proof}[Proof of Lemma~\ref{lem:duality}]
Since $a/\bar g=r/\bar g+q$ with $q=(\bar\mu+u)/\bar g$, the integrating factor in
\eqref{eq:lambda_red_formula} splits as
$\exp(-\int_\xi^s a/\bar g)=e^{-r\tau(\xi,s)}\exp(-\int_\xi^s q)$. Using
$x^*(\xi)=\frac{p}{\bar g(\xi)}\exp(-\int_{l_0}^\xi q)$,
\[
\bar g(\xi)\psi_r(\xi)x^*(\xi)
=p\,e^{-\int_{l_0}^\xi q}\!\int_\xi^{l_m}\frac{\chi(s)}{\bar g(s)}
e^{-r\tau(\xi,s)}e^{-\int_\xi^s q}\dd s
=\int_\xi^{l_m}\chi(s)\,e^{-r\tau(\xi,s)}\underbrace{\frac{p}{\bar g(s)}
e^{-\int_{l_0}^s q}}_{=\,x^*(s)}\dd s
=W_r(\xi).
\]
The forward survival $l_0\!\to\!\xi$ and backward survival $\xi\!\to\!s$ telescope
to $l_0\!\to\!s$; the identity uses only the closed forms and so holds for
bang--bang $u$.
\end{proof}

\begin{proof}[Proof of Theorem~\ref{thm:closure_gain_identity}]
Substitute $\psi_r'=(a\psi_r-\chi)/\bar g$ into the smooth form
\eqref{eq:stationary_G_smooth}:
\[
g_E\psi_r'-\mu_E\psi_r
=\Big(\frac{g_E a}{\bar g}-\mu_E\Big)\psi_r-\frac{g_E\chi}{\bar g}.
\]
With $a=r+\bar\mu+u$, the algebraic identities
$\frac{g_E(\bar\mu+u)}{\bar g}-\mu_E=-\bar g\,b$ and $-g_E/\bar g=\alpha$ give
$\frac{g_E a}{\bar g}-\mu_E=-\bar g\,b-r\alpha$ and $-g_E\chi/\bar g=\alpha\chi$,
hence
\[
g_E\psi_r'-\mu_E\psi_r=(-\bar g\,b-r\alpha)\psi_r+\alpha\chi.
\]
Multiplying by $x^*$ and integrating,
\[
B(r)
=-\int b\,(\bar g\psi_r x^*)
-r\int\alpha\,\psi_r x^*
+\int\alpha\chi x^*.
\]
By Lemma~\ref{lem:duality}, $\bar g\psi_r x^*=W_r$ and
$\alpha\psi_r x^*=\frac{\alpha}{\bar g}W_r$, giving \eqref{eq:Br_general}. At $r=0$,
$W_0=W_{E^*}$, so $B(0)=\mathsf A(E-\int bW_{E^*}=\mathsf A-\mathsf C=\Phi'(E^*)$
by \eqref{eq:Phi_prime_tail}.

For the limit, change variables from size $s$ to travel time $t=\tau(\xi,s)$,
$\dd s=\bar g(s)\dd t$, in \eqref{eq:Wr}:
\[
W_r(\xi)=\int_0^{\tau(\xi,l_m)}H_\xi(t)\,e^{-rt}\dd t,
\qquad
H_\xi(t):=\chi(s(\xi,t))\,x^*(s(\xi,t))\,\bar g(s(\xi,t)),
\]
where $s(\xi,t)$ is the size reached from $\xi$ after travel time $t$, so
$H_\xi(0)=\chi(\xi)x^*(\xi)\bar g(\xi)$. Under the stated continuity, $H_\xi$ is
continuous at $t=0$ and uniformly bounded, so the approximate-identity (Abelian)
estimate gives $rW_r(\xi)\to H_\xi(0)=\chi(\xi)x^*(\xi)\bar g(\xi)$ as
$r\to\infty$, with $0\le rW_r(\xi)\le\sup H_\xi<\infty$ uniformly in $\xi$. By
dominated convergence,
\[
r\int_{l_0}^{l_m}\frac{\alpha(\xi)}{\bar g(\xi)}W_r(\xi)\dd\xi
\longrightarrow
\int_{l_0}^{l_m}\frac{\alpha(\xi)}{\bar g(\xi)}\chi(\xi)x^*(\xi)\bar g(\xi)\dd\xi
=\mathsf A(E,
\]
while $0\le W_r(\xi)\le(\sup H_\xi)/r\to0$ gives $\int b\,W_r\to0$. Hence
$B(r)\to\mathsf A(E-0-\mathsf A(E=0$.
\end{proof}

\begin{proof}[Proof of Corollary~\ref{cor:nonresonance_consequences}]
(a) is Theorem~\ref{thm:closure_gain_identity} with
$\Phi'(E^*)<0$ and continuity of $r\mapsto B(r)$. (b) follows from
$B(r)\le\mathsf A^*<1$.
\end{proof}

\begin{proof}[Proof of Theorem~\ref{thm:validated_certificate}]
The first two conditions exclude exterior zeros; the derivative enclosure makes
$S$ strictly increasing on $I_k$; the endpoint signs and the intermediate value
theorem give one upward zero in $I_k$.
\end{proof}

\begin{proof}[Proof of Proposition~\ref{prop:validated_multiple_switch}]
The intermediate value theorem gives a zero in each of $(q_1,q_2)$, $(q_2,q_3)$.
\end{proof}

\bibliography{refs}

\end{document}